\documentclass[12pt]{article}
\usepackage[margin=1in]{geometry}

\usepackage{amsmath}
\usepackage{amsthm}
\usepackage{amssymb}
\usepackage{mathrsfs}

\usepackage{relsize}
\usepackage{graphicx}
\usepackage[dvipsnames]{xcolor}
\usepackage{tikz}

\usepackage{enumitem}
\setlist{noitemsep, topsep=4.9pt, partopsep=0pt}
\usepackage[colorlinks, citecolor=red, pagebackref]{hyperref}

\usepackage{tocloft}
\newtheorem{theorem}{Theorem}
\newtheorem{prop}[theorem]{Proposition}
\newtheorem{lemma}[theorem]{Lemma}
\newtheorem{cor}[theorem]{Corollary}
\newtheorem{conj}[theorem]{Conjecture}
\theoremstyle{definition}
\newtheorem{definition}[theorem]{Definition}
\newtheorem{example}[theorem]{Example}
\newtheorem{question}[theorem]{Question}
\theoremstyle{remark}
\newtheorem{remark}[theorem]{Remark}

\numberwithin{theorem}{section}
\numberwithin{equation}{section}

\newcommand{\R}{\mathbb{R}}
\newcommand{\C}{\mathbb{C}}
\newcommand{\Z}{\mathbb{Z}}

\newcommand{\mb}{\mathbb}

\renewcommand{\phi}{\varphi}

\newcommand{\paren}[1]{\left(#1\right)}
\newcommand{\set}[1]{\left\{#1\right\}}
\newcommand{\bracket}[1]{\left[#1\right]}
\newcommand{\abs}[1]{\left\lvert#1\right\rvert}
\newcommand{\norm}[1]{\left\lVert#1\right\rVert}
\newcommand{\inner}[1]{\left\langle#1\right\rangle}

\DeclareMathOperator{\tr}{tr}
\DeclareMathOperator{\cond}{cond}
\DeclareMathOperator{\lcm}{lcm}
\DeclareMathOperator{\supp}{supp}

\begin{document}

\title{Riesz bases of exponentials and multi-tiling in finite abelian groups}

\author{Sam Ferguson, Azita Mayeli and Nat Sothanaphan}

\date{\today}
\maketitle
 
\vspace{-3.5ex}
 
\begin{abstract}
Motivated by the open problem of exhibiting a subset of Euclidean space which has no exponential Riesz basis, we focus on exponential Riesz bases in finite abelian groups. We point out that that every subset of a finite abelian group has such a basis, removing interest in the existence question in this context. We then define tightness quantities for subsets to measure the conditioning of Riesz bases; for normalized tightness quantities, a value of one corresponds to an orthogonal basis, and a value of infinity corresponds to nonexistence of a basis. As an application, we obtain new weak evidence in favor of the open problem by giving a sequence of subsets of finite abelian groups whose tightness quantities go to infinity in the limit. We also prove that the Cartesian product of a set with a finite abelian group has the same tightness quantities as the original set. Lastly, under an additional hypothesis, explicit bounds are given for tightness quantities in terms of a subset's lowest multi-tiling level by a subgroup and its geometric configuration. This establishes a quantitative link between discrete geometry and harmonic analysis in this setting.
\end{abstract}

\vspace{-2.5ex}

\clearpage

{
\hypersetup{linkcolor=blue}
\setcounter{tocdepth}{2}
\tableofcontents
\vspace{1ex}
\listoffigures
}

\clearpage

\section{Introduction}
\label{sec:intro}

This paper develops a framework of what may be called \emph{quantitative spectrality} of subsets of finite abelian groups. We hope that this notion may help settle an unsolved problem concerning existence of Riesz bases of exponentials and lead to further exploration of the relationships between Riesz bases and multi-tiling in the Euclidean setting.

We begin by recalling some definitions that help us motivate our work. A subset $E \subseteq \R^d$ is called \emph{spectral} if there is an exponent set $\Lambda \subseteq \R^d$ such that the exponentials
$$E \ni x \mapsto \exp(2\pi i \lambda \cdot x)$$
for $\lambda \in \Lambda$ form an orthogonal basis for $L^2(E)$.
There is an analogous notion of spectrality for subsets $E$ of a finite abelian group $G$, and it is defined by replacing the above exponentials with restrictions to $E$ of homomorphisms from $G$ into the circle group $\mb{S}^1$, though we will still call these restrictions \emph{exponentials}.

The celebrated conjecture of Fuglede \cite[p.~119]{Fug} states that for a subset $E \subseteq \R^d$ of finite and positive measure, $E$ tiles $\R^d$ if and only if $E$ is spectral. We say that $E$ \emph{tiles} $\R^d$ (by translation) if there exist translates $E+t$, $t \in \R^d$, whose union is $\R^d$, up to a set of measure zero, and whose pairwise intersections are of measure zero. Despite partial progress for some subsets $E$, the conjecture was proven false by Tao \cite{Tao} in 2004, starting from the construction of a counterexample to the corresponding conjecture in finite abelian groups.

Nevertheless, connections exist between tiling and spectrality. In particular, a theorem of Fuglede \cite{Fug} asserts that the conjecture is true if, in the above definitions, the exponent set $\Lambda$ and the set of $t$ defining the translates are required to be lattices. A \emph{lattice} in $\R^d$ is an image of the additive group $\Z^d$ under an invertible linear transformation. Fuglede's theorem thus establishes a relationship between harmonic analysis and lattice or discrete geometry.

As spectral subsets are rare, their applicability has limitations. More common, and still useful, are subsets $E$ for which $L^2(E)$ admits a Riesz basis of exponentials. We say that exponentials $e_n$, $n\geq 1$, form a \emph{Riesz basis} for $L^2(E)$ if their span is dense in $L^2(E)$ and there are constants $0<A\leq B<\infty$ such that, for any complex sequence $(c_n) \in \ell^2$,
\begin{equation}
\label{eq:riesz}
A\sum |c_{n}|^{2}\leq \left\|\sum c_{n}e_{n} \right\|_{L^2(E)}^2\leq B\sum |c_n|^{2}.
\end{equation}
Equivalently, the exponentials $e_{n}$ are a Riesz basis if they form the image of an orthonormal basis $f_n$ of $L^2(E)$ under a bounded, invertible operator $T$. Riesz bases of exponentials generalize the notion of an orthogonal basis of exponentials, which corresponds to the case when $A=B$. If $L^2(E)$ admits a Riesz basis of exponentials as above, we call $E$ a \emph{Riesz set}.

Riesz sets are at least somewhat common: any finite union of ``pixels,'' $[0,1]^d+t$ for some $t \in \Z^d$, is a Riesz set. In fact, any bounded, measurable set that multi-tiles by a lattice is a Riesz set; a particularly simple proof is given by Kolountzakis \cite{Ko}. We say that $E$ \emph{multi-tiles} $\R^d$ at level $k \geq 1$ if translates $E+t$, $t \in \R^d$ exist such that almost every point of $\R^d$ is contained in exactly $k$ of these translates. Tiling means multi-tiling at level $1$.

Despite the ease with which they can be constructed, Riesz sets retain some mystery. For example, no subset of Euclidean space has ever been proved not to be a Riesz set. Olevskii and Ulanovskii \cite[p.~7]{OU} recently drew attention to this problem by emphasizing, as the first question in their collection of harmonic analysis lectures, the conjecture that the unit disc in $\R^2$ is not a Riesz set (see Question \ref{question:olevskii}).

Our goals in this paper are twofold: to make progress on resolving the above conjecture stated by Olevskii and Ulanovskii, and to more deeply understand the relationship between Riesz sets and discrete geometry in the spirit of Fuglede's theorem \cite{Fug} and Kolountzakis's simple proof \cite{Ko}. The latter goal assists us in the former.

To more clearly distinguish between various Riesz bases as we research Riesz sets, we introduce the framework of quantitative spectrality. Looking at the ratio $B/A$ in the bounds \eqref{eq:riesz}, we see that a Riesz basis requires $B/A$ to be finite, while an orthogonal basis requires this ratio to be $1$. Hence, we may think of $B/A$, which we will call the \emph{Riesz ratio} of the Riesz basis, as leading towards a quantitative measure of spectrality. Our hope is that the existence of low Riesz ratios for a set may, in at least some situations, imply that the set has a low multi-tiling level, and vice versa. Our main result is a theorem of this type in the finite abelian group setting (Thm. \ref{thm:introMainRhoUpperBd}). See Section \ref{subsec:contributions} for more details.

If a result of this type is true in the Euclidean setting, and some kind of ``continuity'' holds for the Riesz ratios (see Sec. \ref{subsec:continuityRiesz}), then the unit disc conjecture might be resolved in the following way. Approximate the unit disc by polygons which multi-tile at increasingly high levels. Successive polygons might be such that all of their Riesz ratios become arbitrarily high, which might imply that the unit disc has no finite Riesz ratio, i.e., the unit disc is not a Riesz set.

Nevertheless, such an argument, when naively applied to various concocted approximations, appears to give nonsensical results, e.g., that $[0,1]^2$ is not a Riesz set. Thus, it is unclear how far we are from the correct machinery to run this kind of argument, if it is indeed possible. But we hope that our results will provide a starting place for research in this direction.

\subsection{Our contributions}
\label{subsec:contributions}

We formulate and prove a quantitative relationship between the discrete geometry of subsets of finite abelian groups and the bases used in harmonic analysis on these subsets. Under an additional hypothesis on the subsets, which is always satisfied when the underlying groups are cyclic, we successfully establish such a relationship (Thm. \ref{thm:introMainRhoUpperBd}).

Our first result, obtained by generalizing the 2015 proof of Kolountzakis \cite{Ko} to finite abelian groups, shows that all subsets of finite abelian groups have exponential Riesz bases.

\begin{theorem}[See Cor. \ref{cor:basispairexist} and Sec. \ref{subsec:prelimGeneralFinite}]
\label{thm:introRieszExist}
Let $G$ be a finite abelian group. If $E$ is a nonempty subset of $G$, then $L^{2}(E)$ has a basis consisting of group characters of $G$, that is, $E$ has an exponential Riesz basis.
\end{theorem}

\noindent
This shows that, in a sense, trying to resolve the open problem, Question \ref{question:olevskii}, by imitating the construction of Tao's counterexample to Fuglede's conjecture is doomed to failure. Indeed, the strategy of finding a subset of a finite abelian group with no exponential Riesz basis and then ``lifting" this set to Euclidean space to solve the open problem cannot get off the ground if every subset has an exponential Riesz basis.
See Section \ref{sec:prevLit}.

While the qualitative existence question for exponential Riesz bases in finite abelian groups is thus rendered uninteresting, there remains the possibility that insight into the open problem can be had by shifting to a quantitative viewpoint. To this end, we introduce \emph{tightness quantities} and \emph{normalized tightness quantities} which, loosely speaking, attempt to measure quantitatively how close the set $E$ is to having no exponential Riesz basis. 

For an example, recall the inequalities \eqref{eq:riesz}. If $A_{\text{max}}$ denotes the largest value of $A$ for which the left inequality holds, and $B_{\text{min}}$ denotes the smallest value of $B$ for which the right inequality holds, then we can define the \emph{Riesz ratio} $\rho$ of the Riesz basis $\{f_{i}\}_{i\in I}$ to be $B_{\text{min}}/A_{\text{max}}$. Given an exponential Riesz basis $\Lambda$ of $L^{2}(E)$, where $E$ is a subset of a finite abelian group $G$, we can denote its Riesz ratio by $\rho(E,\Lambda)$. If we define
\begin{equation}
\rho(E)=\inf\{\rho(E,\Lambda):\text{ }\Lambda\text{ is an exponential Riesz basis for }E\},
\label{eq:Rieszratio}
\end{equation}
then $\rho(E)$, which we call the \emph{Riesz ratio of} $E$, is the Riesz ratio of the ``best-behaved" exponential Riesz basis of $E$. Symbolically, the expression ``$\rho(E)=\infty$'' should be interpreted to mean that ``$A_{\text{max}}=0$'' or ``$B_{\text{min}}=\infty$,'' that is, no exponential Riesz basis of $E$ exists. Thus, how small $\rho(E)$ is might be interpreted as measuring how ``tight'' $E$ is, or how far $E$ is from having no exponential Riesz basis; for this reason, we refer to the Riesz ratio \eqref{eq:Rieszratio}, and other similar quantities, as tightness quantities (see Sec. \ref{subsec:prelimTightnessQuants}). When a quantity $Q$ is normalized similarly to $\rho$, so $Q\geq 1$ and ``$Q=\infty$'' corresponds to nonexistence of exponential Riesz bases, then we call $Q$ a normalized tightness quantity (see Sec. \ref{sec:decompose}).

Motivated by the above interpretation of $\rho$ and Question \ref{question:olevskii}, we seek to construct sets $E$ for which $\rho(E)$ is large. Further development of such ideas should help us determine which sets in Euclidean space, if any, should be expected to have no exponential Riesz bases. The first new application of the theorems in this paper is that construction of a sequence of sets in finite abelian groups with arbitrarily large $\rho(E)$ is indeed possible.

\begin{theorem}[Ex. \ref{ex:condtoinfty}]
\label{thm:introRhoDiverges}
There exist a sequence of finite abelian groups $G_{n}$ and a sequence of subsets $E_{n}\subseteq G_{n}$ such that
\begin{equation}
\rho(E_{n})\to\infty
\label{eq:Rieszlimit}
\end{equation}
as $n\to\infty$.
\end{theorem}

\noindent
A specific example can be given by taking $G_{n}=\mathbb{Z}_{n+1}^{2}$ and $E_{n}=\left(\{0\}\times\mathbb{Z}_{n+1}\right)\cup\{(1,0)\}$. We can interpret \eqref{eq:Rieszlimit} as providing a weak form of evidence in favor of the open problem, that is, in favor of the existence of a specific subset of Euclidean space having no exponential Riesz basis. This is due to the intuition that if the sequence of sets $E_{n}$ satisfying \eqref{eq:Rieszlimit} ``converged,'' in some sense, to a specific subset $E$ of Euclidean space, and if $\rho$ were continuous with respect to this convergence, then it would follow that
\begin{equation}
\rho(E)=\lim_{n\to\infty}\rho(E_{n})=\infty,
\label{eq:Rieszcontinuity}
\end{equation}
resolving the open problem.

The existence of sets $E_{n}$ as in \eqref{eq:Rieszlimit} raises further questions of whether we can exhibit a general inequality between $\rho(E)$ and quantities pertaining to the geometry of $E$. A lower bound of this type for $\rho(E)$ might give sufficient conditions for \eqref{eq:Rieszlimit} to hold for a sequence of sets $E_{n}$; this would be of immense interest due to its bearing on the open problem, and might lead to its resolution.

It can be shown that $\rho(E)=1$ if and only if $E$ is a spectral set (Prop. \ref{prop:basicPropSpec}), and this fact suggests that a useful lower bound for $\rho(E)$ would already resolve Fuglede's conjecture for wide classes of sets. As such applications are not at hand today, it is not surprising that the lower bound problem appears quite difficult. Thus, we shift our attention to upper bounds. That is, intuitively speaking, we seek an inequality which shows that ``nice'' values of geometric quantities associated with $E$ imply that there exists a ``nearly-orthogonal'' exponential Riesz basis of $E$.

To get started with such an analysis, given a Riesz basis $\Lambda$ of a set $E$, we form the \emph{Fourier matrix} $M=M(E,\Lambda)$ such that each column consists of the values of an element of $\Lambda$ at the points of $E$; this matrix $M$ is useful because, for instance, $\rho(E,\Lambda)=\left(\cond M\right)^{2}$, where ``$\cond$'' is the condition number; we recall that the latter may be calculated in terms of the smallest and largest singular values of the matrix. By making use of cyclotomic polynomials and the values of symmetric polynomials at roots of unity, we can obtain (crude) estimates such as the following.

\begin{theorem}[Thm. \ref{thm:upperbdcondeb}]
\label{thm:introDetLowerBd}
Let $G$ be a finite abelian group, and let $m$ denote the minimal exponent of $G$. Let $E$ be a subset of $G$ with $n$ elements, and let $\Lambda$ be an exponential Riesz basis of $E$. Then, the inequality
\begin{equation}
|\det M(E,\Lambda)|\geq n^{n(1-\varphi(m))/2}
\label{eq:cyclotomicestimate}
\end{equation}
holds, where $\varphi$ is Euler's totient function.
\end{theorem}

\noindent
Recall that, according to Rotman \cite{Ro}, the minimal exponent $m$ of a finite abelian group may be calculated by finding the least common multiple of the orders of all the elements of the group. According to Myerson \cite{My}, it is suspected that the quantities we are estimating cannot be ``anywhere near as small'' as the tiny right side suggests. However, giving an improved, perhaps non-exponential lower bound is in general an open problem on precisely how small a nonzero sum of roots of unity can be.

By itself, the above Theorem \ref{thm:introDetLowerBd} does not yield an upper bound for $\rho(E)$ in terms of geometric quantities associated with $E$. However, by exploiting such inequalities and combining them with results pertaining to multi-tiling by a subgroup---note that a subgroup here is the analogue of a lattice in Euclidean space---we are able to obtain an upper bound as desired, under an additional hypothesis which requires a further definition to state. The results pertaining to multi-tiling by a subgroup are partly inspired by results we obtained earlier on Cartesian products; we briefly mention the highlights of these results before giving the definitions necessary to state our main theorem.

The first observation about Fourier matrices corresponding to Cartesian products in finite abelian groups is that they are tensor products (see Sec. \ref{subsec:decomposeCartesian}). Indeed, if $E_{1}$ is a subset of a group $H$, with exponential Riesz basis $\Lambda_{1}$, and $E_{2}$ is a subset of a group $K$, with exponential Riesz basis $\Lambda_{2}$, then with the usual ordering of rows and columns, the equation
\begin{equation}
M(E_{1}\times E_{2}, \Lambda_{1}\times\Lambda_{2})=M(E_{1},\Lambda_{1})\otimes M(E_{2}, \Lambda_{2})
\label{eq:tensorproduct}
\end{equation}
gives the Fourier matrix for the Cartesian products.

As the singular values of a tensor product are the products of the singular values of the factors---see, for instance, the reference of Horn and Johnson \cite{HJ}---the equation \eqref{eq:tensorproduct} lets us extract information on tightness quantities of Cartesian products. For example, if we define
\begin{equation}
D(E)=\inf\{|\det M(E,\Lambda)|:\Lambda\text{ is an exponential Riesz basis for }E\},
\label{eq:absolutedeterminantofset}
\end{equation}
(see Def. \ref{def:absoluteDet} and \ref{def:setquantities}) and observe that $|\det M(E,\Lambda)|$ is the product of the singular values of $M(E,\Lambda)$, then this fact becomes relevant. We can use \eqref{eq:absolutedeterminantofset} to show that
\begin{equation}
\widetilde{D}(E_{1}\times E_{2})\leq\widetilde{D}(E_{1})\widetilde{D}(E_{2}),
\label{eq:absdetprodest}
\end{equation}
(Cor. \ref{cor:bdQcartesianset}) where $\widetilde{D}(E)\in [1,\infty)$ is the normalized tightness quantity corresponding to $D(E)$, for each set $E$. The precise definition of $\widetilde{D}(E)$ is identified later, in the body of our paper, using Hadamard's inequality for matrices (see Def. \ref{def:normalizedquants}).

As an application of \eqref{eq:absdetprodest}, or rather the analogous inequality for $\rho$, we obtain a sequence of sets which are, in the limit, as well-behaved as can be hoped for with respect to tightness quantities, despite neither tiling nor being spectral.

\begin{theorem}[Ex. \ref{ex:cartclosespec}]
\label{thm:introRhoGoesToOne}
There exist a sequence of finite abelian groups $G_{n}$ and a sequence of subsets $E_{n}\subseteq G_{n}$ such that
\begin{equation}
1 < \rho(E_{n})\to 1
\label{eq:asymptoticgood}
\end{equation}
as $n\to\infty$.
\end{theorem}

\noindent
For example, if we take $G_{n}=\mathbb{Z}_{p_{n}}^{2}$ and $E_{n}=\{0,1\}\times\mathbb{Z}_{p_{n}}$, where $p_{n}$ is the $n^{\text{th}}$ odd prime, then
\[
\rho(E_{n})\leq \rho(\{0,1\})\rho(\mathbb{Z}_{p_{n}})=\rho(\{0,1\})\to 1
\]
as $n\to\infty$. More precisely, using $\Lambda=\{0,\frac{p-1}{2}\}$ in $\mathbb{Z}_{p_n}$, we can compute that $\rho(\{0,1\}) - 1 \sim \frac{\pi}{p_{n}}$ asymptotically as $n\to\infty$. The contrast between \eqref{eq:asymptoticgood} and \eqref{eq:Rieszlimit} suggests that there is a full range of possible asymptotic behaviors of ``discrete shapes'' to explore, and this behavior may be sensitive to the multi-tiling level of the sets involved.

We also have a noteworthy result that, at least when the second factor of a Cartesian product is a group in its own right, inequalities of the type given by \eqref{eq:absdetprodest} become exact equalities.

\begin{theorem}[Thm. \ref{thm:dimexpand}]
\label{thm:introRhoProdExact}
Let $E$ be a subset of a finite abelian group $H$, and let $K$ be any finite abelian group. Then,
\begin{equation}
\rho(E\times K)=\rho(E).
\label{eq:cartequality}
\end{equation}
In particular, $E$ is spectral (in $H$) if and only if the Cartesian product $E\times K$ is spectral (in $H \times K$).
\end{theorem}

\noindent
The equation \eqref{eq:cartequality} may be viewed as saying that the conditioning of the best-conditioned Riesz basis for a ``cylindric domain'' is precisely equal to the conditioning of the best-conditioned Riesz basis for the ``base'' of the cylindric domain. From this viewpoint, it may be thought of as a generalization of the 2016 result of Greenfeld and Lev \cite{GnL1} that, under certain circumstances, $E$ is spectral if and only if $E\times K$ is spectral, where $K$ is an interval in $\mathbb{R}$ or a convex planar domain in $\mathbb{R}^{2}$. This result was used in their main 2017 theorem \cite{GnL2} that Fuglede's conjecture holds for all convex polytopes in $\mathbb{R}^{3}$. This theorem and the 2003 result of Iosevich, Katz, and Tao \cite{IKT} that Fuglede's conjecture holds for convex planar domains in $\mathbb{R}^{2}$ were superseded when Lev and Matolsci \cite{LM} proved that Fuglede's conjecture holds for convex domains in all dimensions in 2019. Nevertheless, results on Cartesian products remain of interest because the conjecture that $E$ and $K$ must be spectral when $E\times K$ is spectral remains unresolved, according to Greenfeld and Lev \cite{GnL3}.

Finally, we carefully define a new term needed for us to state the main result of our paper; cf. Sec. \ref{subsec:decomposeMultiTile}. Suppose that $E$ multi-tiles a finite abelian group $G$ at level $\ell$ by a translation set $H$ which is a subgroup of $G$. Then, associated with $(E, H)$, we can consider the number of distinct ``cross sections'' of $E$ with respect to $H$ to be a measure of its geometric complexity. More precisely, define
\[
K=G/H=\{k_{1}, \dots, k_{m}\},
\]
so the $k_{i}\subset G$ are cosets. Then, we define the cross sections of $E$ with respect to $H$ to be the sets
\[
F_{i}=E\cap k_{i}.
\]
For (some) $g_{i}\in k_{i}$, we can define
\[
F_{i}'=F_{i}-g_{i}.
\]
Then, $F_{i}'$ is a subset of $H$, which we call a \emph{translated $H$-cross section} of $E$. Although $F_{i}'$ depends on the specific element $g_{i}\in k_{i}$ that we used, $F_{i}'$ is unique up to $H$-translation. Thus, the number $k$ of distinct (equivalence classes of) translated $H$-cross sections of $E$, modulo $H$-translation, is uniquely determined by $E$ and $H$.

Here is our main result.

\begin{theorem}[Thm. \ref{thm:mainUpperBdLevelComplexity}]
\label{thm:introMainRhoUpperBd}
Let $E$ be a subset of a finite abelian group $G$, and let $H$ be a subgroup of $G$. If \begin{enumerate}[label=(\roman*)]
\item $E$ multi-tiles $G$ by $H$ at level $\ell$,
\item $E$ has $k$ distinct translated $H$-cross sections $F_{i}'\subseteq H$, modulo $H$-translation, and 
\item the distinct translated $H$-cross sections $F_{i}'$ of $E$ have a simultaneous exponential Riesz basis as subsets of $H$, 
\end{enumerate}
then
\begin{equation}
\rho(E)<e\ell^{k\ell+1}.
\label{eq:geometricbound}
\end{equation}
($e$ is Euler's number.)
\end{theorem}

It can be shown that, for $H$ a cyclic group, condition (iii) above always holds (Prop. \ref{prop:cyclicsimulbasis}). As a corollary of \eqref{eq:geometricbound}, we know that the behavior exhibited by \eqref{eq:Rieszlimit} cannot occur under certain circumstances:

\begin{cor}
Let $G_{n}$ be a sequence of finite abelian groups. For each $n$, let $E_{n}$ be a subset of $G_{n}$ which
\begin{enumerate}[label=(\roman*)]
\item multi-tiles by a subgroup $H_{n}$ of $G_{n}$ at level $\ell_{n}$,
\item has $k_{n}$ distinct translated $H_{n}$-cross sections modulo $H_{n}$-translation, and
\item the distinct translated $H_{n}$-cross sections have a simultaneous Riesz basis as subsets of $H_n$.
\end{enumerate}
Then, if $\{\ell_n\}_{n=1}^\infty$ and $\{k_n\}_{n=1}^\infty$ are bounded, then $\{\rho(E_{n})\}_{n=1}^{\infty}$ is bounded. In particular,
\[
\rho(E_{n})\not\to\infty
\]
as $n\to\infty$.
\end{cor}

Our paper is organized as follows.

\begin{itemize}
\item Section \ref{sec:prevLit}. We give background on previous literature.

\item Section \ref{sec:prelim}. We give preliminaries and notations. A noteworthy result is Theorem \ref{thm:introRieszExist} on existence of exponential Riesz bases.

\item Section \ref{sec:properties}. We investigate basic (Sec. \ref{subsec:basicProps}) and invariance (Sec. \ref{subsec:invarianceProps}) properties of tightness quantities. Of note are the use of Hadamard's inequality (Rem. \ref{remark:HadamardIneq}), the affine restriction property (Prop. \ref{prop:affinerestrict}), and an example that proves Theorem \ref{thm:introRhoDiverges}.

\item Section \ref{sec:bounds}. We prove rather sophisticated estimates for tightness quantities, starting from their relationships (Sec. \ref{subsec:boundsRelations}), upper bounds for pairs (Sec. \ref{subsec:boundsUpperEB}), and finally upper bounds for sets (Sec. \ref{subsec:boundsUpperE}).
A key estimate is Lemma \ref{lem:lowerbdrootunity}, leading to Theorem \ref{thm:introDetLowerBd} for pairs. The argument later extends to Lemma \ref{lem:lowerbdrootunitykB} and Theorem \ref{thm:upperBdSetBySize} for sets.

\item Section \ref{sec:decompose}. We prove various decomposition results that will be used in the proof of our main theorem. Section \ref{subsec:decomposeCartesian} derives inequalities involving Cartesian products. Section \ref{subsec:decomposeMultiTile} concerns multi-tiles by subgroups (Thm. \ref{thm:decomposeMultiTile}), generalizing Cartesian products where one factor is the full direct summand. Finally, Section \ref{subsec:decomposeExact} proves an exact result of independent interest, Theorem \ref{thm:introRhoProdExact}.

\item Section \ref{sec:MultiTileGeometricComplex}. We investigate simultaneous bases (Sec. \ref{subsec:simulBasis}) and state and prove our main Theorem \ref{thm:introMainRhoUpperBd} (Sec. \ref{subsec:mainResult}).

\item Section \ref{sec:conj}. We discuss a heuristic that any continuity we have in the tightness quantities may be ``going the wrong way'' for the open problem to be as easily resolved as \eqref{eq:Rieszcontinuity} suggests (Sec. \ref{subsec:continuityRiesz}), as well as a conjecture on fractals (Sec. \ref{subsec:fractals}).
\end{itemize}

During our investigations, we have made used of a MATLAB program we wrote ourselves to compute tightness quantities. It is available at \url{https://github.com/natso26/RieszMultiTiling}.

 \subsection{Open questions}
\label{subsec:openQuestions}

There are too many open problems that arise from this work for us to give any definitive list, but here are four that we have found most fascinating.

\begin{question}[Simultaneous basis question]
Is it true that for every finite abelian group $G$, and every pair of nonempty subsets $E_{1}, E_{2}$ of $G$ such that $|E_{1}|=|E_{2}|$, there exists a collection of exponentials $\Lambda\subseteq\widehat{G}$ such that the restrictions of the exponentials to $E_{1}$ form a Riesz basis for $E_{1}$ and the restrictions to $E_{2}$ form a Riesz basis for $E_{2}$? In other words, given two subsets of the same size, is there always a simultaneous exponential Riesz basis for both?
\end{question}

In $\mathbb{Z}_{2}^{2}$, there exist three subsets, each with two elements, which have no simultaneous Riesz basis (Ex. \ref{ex:nosimulbasis}). However, the above question for two subsets is unresolved.

\begin{question}[Lower bound for $\rho$]
Can we obtain a nontrivial lower bound for $\rho(E)$, perhaps under additional hypotheses, in terms of geometric information about $E$?
\end{question}

We might hope for some inequality similar in character to our main Theorem \ref{thm:introMainRhoUpperBd}.

\begin{question}[Low multi-tiling level but geometrically complex]
If $G_{n}=\mathbb{Z}_{p_{n}}^{2}$, where $p_{n}$ is the $n^{\text{th}}$ odd prime, and we define
\[
E_{n}=\left(\{0\}\times\mathbb{Z}_{p_n}\right)\cup\left((1,1)\mathbb{Z}_{p_n}\right)\cup\{(1,0)\},
\]
then is it true that
\[
\rho(E_{n})\to\infty
\]
as $n\to\infty$?
\end{question}

This question is important because, unlike the family of sets $E_{n}$ for which \eqref{eq:Rieszlimit} occurred before, each set in this sequence mulit-tiles at level $2$. In other words, we know that ``bad'' behavior can be generated with a sequence of sets whose lowest multi-tiling levels become arbitrarily large. Is it possible to generate equally ``bad'' behavior with a sequence of sets whose ``geometric complexity'' $k$ becomes arbitrarily large, but whose lowest multi-tiling levels are bounded? If the answer is ``Yes,'' then this may suggest an alternative pathway to resolving the open question about nonexistence of Riesz bases in Euclidean space, by focusing on geometric complexity rather than the failure of multi-tiling at finite level for the geometric reasoning behind the proposed solution.

More generally, given a family of sets $E_{n}$ as above which are ``defined similarly'' for groups $G_{n}=\mathbb{Z}_{p_n}^{2}$, what are the asymptotics of $\rho(E_{n})$ as $n\to\infty$? Can these asymptotics be given in terms of geometric quantities associated with $E_{n}$?

\begin{question}[Continuous settings]
Can our discrete main Theorem \ref{thm:introMainRhoUpperBd} be generalized to the continuous settings in some way? In particular, our argument has three main components:

\begin{enumerate}
\item Decomposition (Thm. \ref{thm:decomposeMultiTile});
\item Looping around (Prop. \ref{prop:lowQsimul});
\item Simultaneous basis.
\end{enumerate}

Can these three components be thus generalized?
\end{question}

\section{Previous literature}
\label{sec:prevLit}

To introduce and follow the historical thread for the problem we consider here, we first review Fuglede's conjecture \cite[p.~119]{Fug}. In 1974, Fuglede conjectured, for subsets of Euclidean space, a connection between their discrete geometry and orthogonal bases of complex exponentials of the type used in Fourier analysis. Specifically, given a measurable subset $E$ of $\mathbb{R}^{d}$ of finite, positive measure, \emph{Fuglede's conjecture} says that $E$ tiles $\mathbb{R}^{d}$ if and only if $E$ is spectral. We recall that $E$ \emph{tiles} $\mathbb{R}^{d}$ if there exists a translation set $T\subseteq\mathbb{R}^{d}$ such that the translates $E+t=\{x+t:x\in E\}$ of $E$ by the elements $t\in T$ are pairwise disjoint, meaning $(E+t)\cap(E+t')=\varnothing$ for $t\neq t'$, and cover $\mathbb{R}^{d}$, meaning $\bigcup_{t\in T}(E+t)=\mathbb{R}^{d}$, with each of these equations being considered only up to sets of measure zero. We say that $E$ is \emph{spectral} if there exists an exponent set $\Lambda\subseteq\mathbb{R}^{d}$ such that the exponentials $x\mapsto e^{2\pi i \lambda\cdot x}$ with domain $E$, with exponents given by the $\lambda\in\Lambda$, are pairwise orthogonal with respect to $\langle f, g\rangle = \int_{E}f\overline{g}\ dx$ and complete in $L^{2}(E)$; equivalently, $\{x\mapsto e^{2\pi i \lambda\cdot x}\}_{\lambda\in\Lambda}$ forms an orthogonal basis for $L^{2}(E)$. Fuglede \cite[pp.~107--108]{Fug} proved that if $E$ is spectral with exponent set $\Lambda$, then $\Lambda$ is the joint spectrum of a unique family of self-adjoint, pairwise commuting extensions of the operators $\left.\frac{1}{2\pi i}\frac{\partial}{\partial x_{1}}\right|_{C_{c}^{\infty}(E)},\dots, \left.\frac{1}{2\pi i}\frac{\partial}{\partial x_{d}}\right|_{C_{c}^{\infty}(E)}$, hence the name ``spectral" and the common parlance of calling $\Lambda$ a \emph{spectrum} of $E$.

Despite partial results in favor of Fuglede's conjecture, such as the 2001 result of \L aba \cite{Lab} that it holds for a union of two intervals in $\mathbb{R}$ and the 2003 result of Iosevich, Katz, and Tao \cite{IKT} that it holds for convex planar domains in $\mathbb{R}^{2}$, later improved to all dimensions by Lev and Matolcsi \cite{LM} in 2019, the conjecture was disproved by Tao \cite{Tao} in 2004. Specifically, Tao gave a counterexample in $\mathbb{R}^{5}$. To do this, Tao produced a spectral subset $E$ of size $6$ in $\mathbb{Z}_{3}^{5}$, verifying the spectral property by hand. As $6$ does not divide $3^{5}$, this set $E$ does not tile, and hence violates the analogue of Fuglede's conjecture in $\mathbb{Z}_{3}^{5}$. Then, Tao ``lifted'' $E$ to Euclidean space using the function $f:\mathbb{R}^{5}\to \mathbb{Z}_{3}^{5}$, given by $f(x)=r$, which rounds each coordinate $x_{i}$ down to the integer $\lfloor x_{i}\rfloor$ and then finds that integer's remainder $r_{i}$ modulo $3$ in $\mathbb{Z}_{3}$. The inverse image $f^{-1}(E)$ under $f$ of Tao's set $E$, however, does not meet Fuglede's requirement of having finite measure. Nevertheless, using a scaling argument involving four scales, Tao proved that the intersection $C\cap f^{-1}(E)$ of the inverse image with a suitably large and well-chosen cube $C$ in $\mathbb{R}^{5}$ persists in being a spectral set which does not tile; Tao concluded that Fuglede's conjecture fails in $\mathbb{R}^{d}$ for $d\geq 5$.

Tao's counterexample only shows that the ``spectral implies tiling" implication in Fuglede's conjecture fails. Subsequently, the ``tiling implies spectral'' implication has been disproved in $\mathbb{R}^{5}$, and both implications have been disproved in $\mathbb{R}^{3}$ and $\mathbb{R}^{4}$ also, due to the work of Farkas, Kolountzakis, Matolcsi, M\'{o}ra, and R\'{e}v\'{e}sz \cite{FarkMM, FarkR, KM2, KM1, Mat} in 2005 and 2006. All of their counterexamples were obtained by finding and lifting suitable sets from finite abelian groups. Both implications of Fuglede's conjecture in $\mathbb{R}^{d}$ for $d\leq 2$ remain unresolved, but in 2017 it was shown by Iosevich, Mayeli, and Pakianathan \cite{IMP} that the analogue of Fuglede's conjecture in $\mathbb{Z}_{p}^{d}$ holds for all primes $p$ when $d\leq 2$. This is contrasted by the result of Ferguson and Sothanaphan \cite{FS}, given in 2019, that it fails already in $\mathbb{Z}_{p}^{4}$ for all odd primes $p$. The work of Iosevich, Mayeli, and Pakianathan suggests, at least, that any counterexamples to Fuglede's conjecture in $\mathbb{R}$ or $\mathbb{R}^{2}$, if they exist, will be rather more difficult to obtain than the one given by Tao for $\mathbb{R}^{5}$.

In 2016, Olevskii and Ulanovskii drew attention to the following question, by stating it as the first open problem in their book of harmonic analysis lectures \cite[p.~7]{OU}.

\begin{question}[Olevskii, Ulanovskii]
\label{question:olevskii}
Let $\mathcal{B}$ be the unit disc in $\mathbb{R}^{2}$. Does the space $L^{2}(\mathcal{B})$ admit an exponential Riesz basis?
\end{question}

\noindent
Moreover, so far no single example of a set $S$ (in any dimension) is known such that the space $L^{2}(S)$ does not admit an exponential Riesz basis.

To make this question precise, we define a \emph{Riesz basis} of a Hilbert space $\mathcal{H}$ to be a family of vectors $\{f_{i}\}_{i\in I}$ in $\mathcal{H}$ such that (i) every vector $f\in\mathcal{H}$ can be expressed as $f=\sum_{i\in I}c_{i}f_{i}$, for some uniquely determined $c_{i}\in\mathbb{C}$, and (ii) there exist constants $0<A\leq B<\infty$ such that, for all $c_{i}\in\mathbb{C}$, the bounds
\begin{equation}
A\sum_{i\in I}|c_{i}|^{2}\leq \left\|\sum_{i\in I}c_{i}f_{i} \right\|_{\mathcal{H}}^2\leq B\sum_{i\in I}|c_{i}|^{2}\label{eq:Riesz}
\end{equation}
hold. An equivalent definition is that $\{f_{i}\}_{i\in I}$ is a Riesz basis of a Hilbert space $\mathcal{H}$ if it is the image of an orthonormal basis of $\mathcal{H}$ under a bounded, invertible operator $T:\mathcal{H}\to\mathcal{H}$. Given a measurable subset $S$ of finite, positive measure in $\mathbb{R}^{d}$, we define an \emph{exponential Riesz basis of $S$} to be a family of exponential functions $\{e_{\lambda}\}_{\lambda\in\Lambda}$ on $S$, indexed by $\Lambda\subseteq\mathbb{R}^{d}$ and given by $e_{\lambda}(x)=e^{2\pi i \lambda\cdot x}$ for each $\lambda\in\Lambda$, which form a Riesz basis of $L^{2}(S)$; if an exponential Riesz basis of $S$ exists, then we say that $L^{2}(S)$ \emph{admits an exponential Riesz basis}. In the literature, such as in the work of Young \cite{You}, a series of the form $\sum_{\lambda\in\Lambda}c_{\lambda}e_{\lambda}$ is referred to as a \emph{nonharmonic Fourier series}.

It is noteworthy that Olevskii and Ulanovskii first ask, in their open problem, whether the unit disc has an exponential Riesz basis. Already in 1974, Fuglede \cite[pp.~111--112]{Fug} announced the result that there do not exist infinitely many pairwise orthogonal exponentials on the unit disc $\mathcal{B}$. As $L^{2}(\mathcal{B})$ is not finite-dimensional, it follows that $\mathcal{B}$ is not a spectral set in $\mathbb{R}^{2}$. In 2001, Fuglede \cite{Fug2} gave the details of his proof, using Bessel functions to show that in $\mathbb{R}^{d}$, for all $d\geq 2$, there do not exist infinitely many pairwise orthogonal exponentials on the unit ball. Thus, the unit disc in $\mathbb{R}^{2}$, a set having no orthogonal basis of exponentials, is perhaps a reasonable candidate for possessing the stronger property of having no exponential Riesz basis; at least, no known theorem rules out this possibility.

Although direct evidence for the unit disc having no exponential Riesz basis is hard to come by in the literature, this is an area of active research interest, as suggested by a recent result of Iosevich, Lai, Liu, and Wyman proved in 2019 \cite{ILLW}. It implies that $L^2(\mb{S}^{d-1})$ has no ``Fourier frame'' or frame of exponentials. Here, the unit sphere $\mb{S}^{d-1}$ in $\mathbb{R}^{d}$, $d\geq 2$, is equipped with its usual surface measure. Having no Fourier frame implies the weaker property of having no exponential Riesz basis. Thus, this is perceived as relevant to the exponential Riesz basis problem for the unit disc, due to the circle $\mb{S}^{1}$ being its boundary, and appears to support the reasonableness of the unit disc as a candidate for having no exponential Riesz basis under Lebesgue measure.

In contrast, a regular polygon is not necessarily a reasonable candidate for the property of having no Riesz basis of exponentials. For example, it is well known that a regular octagon does not tile $\mathbb{R}^{2}$, unlike a square or a regular hexagon. By the result of Iosevich, Katz and Tao \cite{IKT}, it follows that a regular octagon is not a spectral set. However, it can be shown that a regular octagon does multi-tile $\mathbb{R}^{2}$.

We say that a subset $E$ of $\mathbb{R}^{d}$ \emph{multi-tiles} if there exists a positive integer $\ell$ such that, for some translation set $T\subseteq\mathbb{R}^{d}$, the indicator functions $1_{E+t}$ of the translates of $E$ by the elements $t\in T$ are such that
\begin{equation}
\sum_{t\in T}1_{E+t}(x)=\ell\label{eq:multi}
\end{equation}
holds for almost every $x\in\mathbb{R}^{d}$. If equation \eqref{eq:multi} holds for an integer $\ell$, we also say that $E$ \emph{multi-tiles at level} $\ell$; clearly, a set multi-tiles at level $1$ precisely when it tiles.

Returning to the octagon, which multi-tiles $\mathbb{R}^{2}$, our attention is drawn to the 2014 result of Grepstad and Lev \cite{GL} that if $E$ is a bounded, Riemann measurable set of positive measure that multi-tiles $\mathbb{R}^{d}$ by a translation set $T$ which is a lattice, then $E$ has an exponential Riesz basis. For our purposes, we may define a \emph{lattice} in $\mathbb{R}^{d}$ to be any set which is the image of $\mathbb{Z}^{d}$ under some invertible linear transformation $A:\mathbb{R}^{d}\to\mathbb{R}^{d}$. As it can be shown that a regular octagon multi-tiles by a lattice, and is Riemann measurable since its boundary has Lebesgue measure zero, it follows that a regular octagon has an exponential Riesz basis. Thus, a regular polygon, even if it is not spectral, may have an exponential Riesz basis.

Grepstad and Lev \cite[p.~2]{GL} say that their result is in the ``spirit of Fuglede's theorem." The theorem of Fuglede \cite{Fug} that their remark refers to is below.

\begin{theorem}[Fuglede]
If a subset $S$ of $\mathbb{R}^{d}$ of finite, positive measure tiles by a translation set $T$ which is a lattice, then $S$ admits a spectrum $\Lambda$ which is a lattice, and conversely.
In particular, if $S$ has a translation set given by $T=A(\mathbb{Z}^{d})$, with $A$ a linear transformation having $\det(A)\neq 0$, then the dual lattice $\Lambda=(A^{-1})^{t}(\mathbb{Z}^{d})$ is a spectrum of $S$, and conversely.
\end{theorem}

\noindent
In a 2004 survey, Kolountzakis \cite{KoS} gave a new, brief proof of Fuglede's theorem using tempered distributions and the Poisson summation formula.
Observe that Fuglede's conjecture can be thought of as the generalization of Fuglede's theorem obtained by dropping the lattice requirement on $T$ and $\Lambda$. Also note that, unlike the result of Grepstad and Lev, Fuglede's theorem does not require that $S$ be Riemann measurable.

In 2015, Kolountzakis \cite{Ko} gave an independent proof of Grepstad and Lev's theorem, albeit with the hypothesis on $E$ of Riemann measurability weakened to Lebesgue measurability. Moreover, while Grepstad and Lev used the method of Meyer's quasicrystals, Kolountzakis's proof is essentially elementary in character, using only linear algebra, the homomorphism property of exponentials, and basic measure theory. This raises the question of whether the study of Riesz bases of exponentials might be more clearly elaborated in the more general context of locally compact abelian groups. This question serves as one starting point of our paper. Another motivation comes from Tao's success in addressing Fuglede's conjecture by examining finite abelian groups. Based on these considerations, it seems natural to examine Riesz bases of exponentials in finite abelian groups, in the hopes of shedding light on the open problem of whether there exists a set which has no exponential Riesz basis.

\section{Preliminaries and notations}
\label{sec:prelim}

For a vector $x\in \C^d$, we denote the $\ell_2$ norm by $\|x \|$ and define it by $\|x\| := \sqrt{\sum_i |x_i|^2}$. The operator norm or induced-$\ell_2$ norm of a matrix $A$ is defined by 
$$\|A\|=\sup\{\|Ax\|: x\in \C^d, \  \ \|x\|=1\}.$$

\begin{definition}[Condition number]
The condition number of a matrix $A$ is given by $\cond (A)=\norm{A}\norm{A^{-1}}$. We define this number to be $\infty$ if $A$ is not invertible.
\end{definition}  
 
It is clear that the condition number of a matrix is at least $1$, as $1=\|A A^{-1}\|\leq\cond(A)$ for any invertible matrix $A$. Observe also that $\cond(A)=\sigma_{\max}(A)/\sigma_{\min}(A)$, where $\sigma_{\max}(A)$ and $\sigma_{\min}(A)$ denote the maximum and minimum of the set of singular values of $A$, respectively. This also justifies that $\cond(A)\geq 1$ and $\cond(A)=\infty$ if $A$ is not invertible.

We explain our setup in $\Z_m^d$ first, and then generalize to arbitrary finite abelian groups later in Section \ref{subsec:prelimGeneralFinite}.

\subsection{Exponential Riesz bases}
\label{subsec:prelimRieszBases}

Let $m$ and $d$ be positive integers, and consider $\Z_m^d$. 
Define $\chi(t)=\exp(2\pi i t)$. 
For nonempty $E,B \subseteq \Z_m^d$ of equal size, we call $(E,B)$ an \emph{equal-size pair}. For any $b\in B$, we define  exponential function $\chi_b: E\to  \mb S^1$  by 
$\chi_b(x):=\chi(\inner{x,b}/m)$, $x\in E$. Here, $\inner{x,b}$ is the usual inner product in $\mb Z_m^d$. 

\begin{definition}
We call $(E,B)$ a \emph{basis pair} if the exponential functions $\chi(\inner{\cdot,b}/m)$, $b \in B$, form a basis of $L^2(E)$.
\end{definition} 

It is known that in finite-dimensional space, every basis is a Riesz basis. Nevertheless, we are interested in the Riesz constants of subfamilies of the functions $\chi(\inner{\cdot,b}/m)$, $b \in B$, when they form a basis. This leads to the following definition.

\begin{definition}
\label{def:riesz}
Let $(E,B)$ be an equal-size pair with $\abs{E}=n$.
Let $B=\set{b_1,\dots,b_n}$, and $f_i: = \chi(\inner{\cdot,b_i}/m)$  be  functions  defined on $E$.
We say $(E,B)$  is a Riesz basis pair if there are positive and non-zero  constants $A_1$ and $A_2$ such that the following holds for all finite sequences $\{c_i\}_{i=1}^n \in \C^n$
\begin{equation}
\label{eq:rieszdef}
A_1 \sum_{i=1}^n \abs{c_i}^2 \leq \norm{\sum_{i=1}^n c_i f_i}_{L^2(E)}^2 \leq A_2 \sum_{i=1}^n \abs{c_i}^2 . 
\end{equation}

In this case, we say the exponentials $\{\chi(\inner{\cdot,b}/m): \ b\in B\}$ is a Riesz basis for $L^2(E)$ or $E$ has a exponential Riesz basis. \\

We define the \emph{optimal lower Riesz constant} $L_E(B)$ and \emph{optimal upper Riesz constant} $U_E(B)$ to be, respectively, the supremum of all constants $A_1$ and the infimum of all constants  $A_2$ such that the inequality 
\eqref{eq:rieszdef}
holds for   $\{c_i\}_{i=1}^n \in \C^n$.
\end{definition}

Notice that $0 \leq L_E(B) \leq U_E(B) < \infty$.
If $(E,B)$ is a basis pair, then since the corresponding basis is Riesz, $L_E(B) > 0$ as well. \\

An interesting question  to ask here is  if any set $E$ has an exponential Riesz basis. 
Later in Corollary \ref{cor:basispairexist} we will prove that the answer to this question is affirmative and for every $E$, there is a $B$ such that $(E,B)$ is a basis pair.  It is easy to see that the optimal  upper and lower Riesz constants are identical if in addition the orthogonality holds. First we have a definition. 

\begin{definition}
We call $(E,B)$ a \emph{spectral pair} if the exponential functions $\chi(\inner{\cdot,b}/m)$, $b \in B$, are orthogonal on $E$. 
\end{definition} 

If $(E,B)$ is a spectral pair with $\abs{E}=n$, then since the functions $f_i$ are orthogonal and have norm $\sqrt{n}$, $L_E(B) = U_E(B) = n$. It is known that not every set $E\subseteq \Z_m^d$ is a spectral set. 
For example, the spectral sets in $\Z_p^2$ are of size $1, p, p^2$ \cite{IMP}. 
We are interested in such a set $E$, given by the following definition.

\begin{definition}
A set $E$ is a \emph{spectral set} if there is a $B$ such that $(E,B)$ is a spectral pair. In this case, $B$ is called a  \emph{spectrum set} for $E$.
\end{definition}

\subsection{Tightness quantities}
\label{subsec:prelimTightnessQuants}

We define quantities that will help us to understand the behaviors of the Riesz constants $L_E(B)$ and $U_E(B)$.
Let $E=\set{e_1,\dots,e_n}$ and
$B=\set{b_1,\dots,b_n}$, indexed according to the lexicographic order in $\Z_m^d$. That is, $i< j$ implies there exists $k$ such that $(e_i)_{\ell}=(e_j)_{\ell}$ for $\ell<k$ and $(e_i)_{k}<(e_j)_{k}$, and similarly for the elements of $B$. (We could pick any other order, as the quantities we are going to define do not depend on the orderings of $E$, $B$.)

\begin{definition}[Fourier matrix]
Let $(E,B)$ be an equal-size pair with $\abs{E}=n$.
The \emph{Fourier matrix} of $(E,B)$, denoted  by $T(E,B)=(t_{i,j})$, is the $n\times n$ matrix with the  entries  $t_{i,j}:=\chi( \inner{e_i,b_j}/m), 1\leq i,j\leq n$. 
\end{definition}

Note that the $j$th column of $T(E,B)$ is the function $f_j = \chi(\inner{\cdot,b_j}/m)$ applied to each element of $E$.  With this, we can write 

$$ T(E,B) =  \left(\begin{array}{c}
\vdots \\ f_1 \\ \vdots 
\end{array}
\begin{array}{c}
\vdots \\ f_2 \\ \vdots 
\end{array}
\cdots 
\begin{array}{c}
\vdots \\ f_n \\ \vdots
\end{array} \right).
$$

The following result presents the  exact values of  optimal Riesz constants  of a pair in terms of operator norm of Fourier matrix. 

\vskip.124in

\begin{prop}
\label{prop:exprrieszconst}
Let $(E,B)$ be an equal-size pair. Let $T=T(E,B)$ be the associated Fourier matrix. Then
$$L_E(B) = \norm{T^{-1}}^{-2}, \ \  \text{and} \ \   U_E(B) = \norm{T}^2,$$
where $L_E(B) = 0$ if $T$ is not invertible. Moreover, 
\begin{equation}\label{u-t} 
 U_E(B) =\cond(T)^2 L_E(B).
\end{equation}
\end{prop}

\begin{proof} Notice the 
equation \eqref{eq:rieszdef} can be written as
$$A_1 \norm{c}^2 \leq \norm{Tc}^2 \leq A_2 \norm{c}^2$$
for all $c \in \C^n$.
Thus, the minimum possible $A_2$ is
$$\sup_{c \neq 0} \frac{\norm{Tc}^2}{\norm{c}^2} = \norm{T}^2.$$
On the other hand, if $T$ is not invertible, then the maximum possible $A_1$ is zero. Otherwise, this value is
$$\inf_{c \neq 0} \frac{\norm{Tc}^2}{\norm{c}^2}
= \inf_{c \neq 0} \frac{\norm{c}^2}{\norm{T^{-1}c}^2} =
\paren{\sup_{c \neq 0} \frac{\norm{T^{-1}c}^2}{\norm{c}^2}}^{-1}
= \norm{T^{-1}}^{-2}.$$
 
 The equation (\ref{u-t}) now holds from the relation 
 
$$\sqrt{U_E(B)/L_E(B)} = \norm{T}\norm{T^{-1}}=\cond(T) .$$
This completes the proof. 
\end{proof}

Notice, according to \eqref{u-t} the quantity  $\cond(T)$  captures how ``tight'' the Riesz basis is, with the tightest possible basis having $L_E(B)=U_E(B)$. This will only hold if $\cond(T)=1$ or $\|T\|= \|T^{-1}\|^{-1}$. Later, in Proposition \ref{prop:basicspectral}, we will prove  that this is only possible when the pair is a spectral pair, i.e, the orthogonality also holds. \\

We define a number measuring this tightness as follows.

\vskip.2in

\begin{definition}
\label{def:cond}
The \emph{Riesz ratio of $B$ with respect to $E$} is
$$\rho_E(B)  = \cond(T(E,B))^2 = U_E(B)/L_E(B).$$
\end{definition}

We are also interested in the absolute value of the determinant of $T(E,B)$.
It will be important later in obtaining bounds on $\rho$, $U$ and $L$ (see Sec. \ref{sec:bounds}).

\begin{definition}
\label{def:absoluteDet}
The \emph{absolute determinant of $B$ with respect to $E$} is $D_E(B) = \abs{\det T(E,B)}$.
\end{definition}

As $T(E,B)$ becomes more singular, $\rho_E(B)$ goes to infinity, while $D_E(B)$ goes to zero. Thus, $D_E(B)$ offers another way to understand how singular $T(E,B)$ is, and this corresponds to how tight the Riesz basis is, with less tight bases having more singular $T(E,B)$.
Because $D_E(B)$ is the absolute value of a linear combination of $m$th roots of unity, it is sometimes easier to work with than $\rho_E(B)$.

Later, we will show that the quantities $L_E(B)$, $U_E(B)$, $\rho_E(B)$, $D_E(B)$ each provide a way of quantifying how ``close" to being spectral a pair $(E,B)$ is.
By looking at the optimizing partner $B$, we can define the following quantities for a set $E$ which, as we later show, capture the notion of $E$ being ``close" to spectral.

\begin{definition}
\label{def:setquantities}
\
\begin{itemize}
\item The \emph{lower Riesz constant} of $E$ is $L(E) = \max_B L_E(B)$.
\item The \emph{upper Riesz constant} of $E$ is $U(E) = \min_B U_E(B)$.
\item The \emph{Riesz ratio} of $E$ is $\rho(E)=\min_B \rho_E(B)$.
\item The \emph{absolute determinant} of $E$ is $D(E) = \max_B D_E(B)$.
\end{itemize}
\end{definition}

\subsection{Existence of Riesz bases}
\label{subsec:prelimExistence}

As mentioned earlier, for every $E$, there is a $B$ such that $(E,B)$ is a basis pair. Here, we prove that fact, by first proving a more general proposition.

\begin{prop}[Basis restriction]
\label{prop:basisrestrict}
Let $S$ be a finite set, and let $m=|S|$. If $\{f_{1},\dots, f_{m}\}$ is a basis for $L^{2}(S)$, then, given a nonempty subset $E$ of $S$, there exists a subset of $\{f_{1}, \dots, f_{m}\}$ of size $|E|$ which is a basis for $L^{2}(E)$.
\end{prop}

\begin{proof}
Observe that $\text{dim}\left(L^{2}(S)\right)=m$, and let $n=|E|$. Write $E=\{x_{1},\dots, x_{n}\}$ and write $S=\{x_{1}, \dots, x_{m}\}$, so $m\geq n$. Let $N=(n_{i,j})_{1\leq i, j\leq m}$ be the $m\times m$ matrix given by $n_{i,j}=f_{i}(x_{j})$. As $\{f_{1},\dots, f_{m}\}$ is a basis for $L^{2}(S)$, the rank of $N$ is $m$. Hence, the $m$ columns of $N$ are linearly independent. In particular, the first $n$ columns of $N$ are linearly independent. Thus, the rank of the $m\times n$ matrix $\tilde{N}=(n_{i,j})_{1\leq i\leq m, 1\leq j\leq n}$ consisting of the first $n$ columns of $N$ is $n$. As column rank equals row rank, there exist $n$ linearly independent rows of $\tilde{N}$; say $\{\left(f_{i_{j}}(x_{1}), \dots, f_{i_{j}}(x_{n})\right)\}_{j=1}^{n}$ is such a set of $n$ linearly independent rows. Then $\{f_{i_1},\dots, f_{i_n}\}$, when restricted to $E$, is a basis for $L^{2}(E)$.
\end{proof}

\begin{cor}
\label{cor:basispairexist}
Let $E \subseteq \Z_m^d$. Then there is a $B$ such that $(E,B)$ is a basis pair.
\end{cor}

\begin{proof}
The exponential functions $\chi(\inner{\cdot,b}/m)$, $b \in \Z_m^d$, form a basis of $L^2(\Z_m^d)$. The result now follows from Proposition \ref{prop:basisrestrict}
\end{proof}

We now give an alternative proof of Corollary \ref{cor:basispairexist}.

\begin{proof}[Alternative proof] Let $n=|E|$. Then, $L^{2}(E):=\{f:E\to\mathbb{C}\}$ is $n$-dimensional, so we seek $n$ exponentials to serve as a basis. That is, we want to find $a_{1}, \dots, a_{n}$ in $\Z_m^{d}$ such that, given $f\in L^{2}(E)$, there exist unique $c_{1}, \dots, c_{n}\in\mathbb{C}$ such that
\[
f(x)=\sum_{j=1}^{n}c_{j}e_{a_j}(x)
\]
for all $x\in E$, where $e_{a_j}(x)=e^{2\pi i (a_{j}\cdot x)/m}$. Writing $E$ as $E=\{x_{1}, \dots, x_{n}\}$, we see that this is equivalent to seeking $a_{1}, \dots, a_{n}$ such that the system of $n$ equations
\[
f(x_{r})=\sum_{j=1}^{n}c_{j}e_{a_j}(x_{r}),
\]
with $1\leq r\leq n$, in the $n$ unknowns $c_{1}, \dots, c_{n}$, has a unique solution for all $\left(f(x_{1}), \dots, f(x_{n})\right)\in\mathbb{C}^{n}$. This, in turn, is equivalent to the existence of $a_{1}, \dots, a_{n}$ such that the matrix of coefficients of the $c_{j}$, denoted $M=\left(m_{j,r}\right)_{1\leq j, r\leq n}$ with $m_{j, r}=e_{a_j}(x_{r})$, satisfies $\det(M)\neq 0$. If, however, no such $a_{1}, \dots, a_{n}$ in $\Z_m^d$ exist, then for all $a=\left(a_{1}, \dots, a_{n}\right)\in\Z_m^{dn}$, we have
\[
0=\det(M)=\sum_{\sigma\in S_{n}}\text{sgn}(\sigma)e(a_{1}\cdot x_{\sigma(1)})\cdots e(a_{n}\cdot x_{\sigma(n)}).
\]
Equivalently, letting $x_{\sigma}=(x_{\sigma(1)}, \dots, x_{\sigma(n)})\in\Z_m^{dn}$, we have
\[
0=\sum_{\sigma\in S_{n}}\text{sgn}(\sigma)e(a\cdot x_{\sigma})
\]
for all $a\in\Z_m^{dn}$, where $e(a\cdot x_{\sigma})=e^{2\pi i (a\cdot x_{\sigma})/m}$, contradicting the fact that the family of functions $\{a\mapsto e(a\cdot x_{\sigma})\}_{\sigma\in S_{n}}$ is linearly independent, because it is a family of $n!$ distinct orthogonal characters of $\Z_m^{dn}$.
\end{proof} 

\subsection{General finite abelian groups}
\label{subsec:prelimGeneralFinite}

The following is a natural setting for some of our results, specifically some invariance properties in Section \ref{subsec:invarianceProps}.

Let $G$ be a finite abelian group.
We call a homomorphism from an abelian group $G$ to an abelian group $H$ a \emph{$\Z$-linear transformation}.
Thus, an invertible $\Z$-linear transformation is the same as a group isomorphism.
Call a composition of a $\Z$-linear transformation followed by a translation an \emph{affine transformation}.

Let $\widehat{G}$ be the Pontryagin dual of $G$, that is, the set of all homomorphisms from $G$ to the circle group $\mb{S}^1$.
We have $G \cong \widehat{G}$, but there is no canonical isomorphism. However, if we pick generators so that $G \cong \Z_{n_1}\times\dots\times \Z_{n_k}$, then an element $b \in \widehat{G}$ acts on each $(x_1,x_2,\dots,x_k) \in G$, $x_i \in \Z_{n_i}$, by
\begin{equation}
\label{eq:ghatacts}
b(x_1,x_2,\dots,x_k) = \chi\paren{\frac{x_1 b_1}{n_1}+\frac{x_2 b_2}{n_2}+\dots+\frac{x_k b_k}{n_k}}
\end{equation}
for some $b_i \in \Z_{n_i}$. So we can identify $b\in \widehat G$ with $(b_1,b_2,\dots,b_k) \in G$.

For the double dual $\widehat{\widehat{G}}$, there is a canonical isomorphism $G \cong \widehat{\widehat{G}}$ that identifies $x \in G$ with the evaluation map $b \mapsto b(x)$.

For nonempty $E \subseteq G$ and $B \subseteq \widehat{G}$ of equal size, we call $(E,B)$ an \emph{equal-size pair}. If $G = \Z_m^d$ and we identify $\widehat{\Z_m^d}$ with $\Z_m^d$ as above, we recover our previous definition in $\Z_m^d$.

Call an equal-size pair $(E,B)$ a \emph{basis pair} if the elements of $B$, after restriction to $E$, form a basis of $L^2(E)$. Call it a \emph{spectral pair} if these elements are orthogonal on $E$.
(Recall that $f_1, f_2 \in L^2(E)$ are orthogonal if $\sum_{x \in E} f_1(x)\overline{f_2(x)} = 0$.)
In such a case, call $E$ a \emph{spectral set} and $B$ a \emph{spectrum set} for $E$. These coincide with our previous definitions for $\mathbb{Z}_{m}^{d}$. Similarly, we can define $L_E(B)$ and $U_E(B)$ by modifying Definition \ref{def:riesz} in the same way.

Let $T(E,B)$ be the linear operator from $L^2(B)$ to $L^2(E)$ defined by
\begin{equation}
\label{eq:TEBgeneral}
T(E,B)h = \sum_{b \in B} h(b) b \in L^2(E), \quad h \in L^2(B).
\end{equation}
Let $E=\set{x_1,x_2,\dots,x_n}$ and $B=\set{b_1,b_2,\dots,b_n}$.
If we identify $f \in L^2(E)$ with the vector $(f(x_1),f(x_2),\dots,f(x_n)) \in \C^n$ and $h \in L^2(B)$ with the vector $(h(b_1),h(b_2),\dots,h(b_n)) \in \C^n$, then $T(E,B)$ is the matrix whose $(i,j)$-th entry is $b_j(x_i)$. So this definition of $T(E,B)$ coincides with our previous definition in the case  $G=\Z_m^d$.
It is easy to check that Proposition \ref{prop:exprrieszconst} still holds.
Moreover, $\rho_E(B)$ and $D_E(B)$ can be defined as before.

Definition \ref{def:setquantities} depends on what ambient group $B$ is allowed to sit inside. For example, if $E$ is a subset of of a subgroup $H$ of $G$, and we optimize over $B\subseteq \widehat{H}$, then we write $L(E;H)$ for the optimal lower Riesz constant.
In general, if the ambient group $G$ containing $E$ needs to be specified, we will write $L(E;G)$ instead of $L(E)$, and similarly for the other tightness quantities.

It is known that the characters of $G$ are linearly independent and $|G|=|\widehat{G}|$,  so $\widehat{G}$ is a basis of $L^2(G)$.
By the basis restriction proposition \ref{prop:basisrestrict}, for every $E \subseteq G$, there is $B \subseteq \widehat{G}$ such that $(E,B)$ is a basis pair. Therefore, all results from Section \ref{sec:prelim} can be generalized to this setting.

In particular, Corollary \ref{cor:basispairexist} becomes Theorem \ref{thm:introRieszExist}.

Similarly, results in Section \ref{subsec:basicProps} hold in this setting.
We give some remarks on duality.
If $(E,B)$ is an equal-size pair, then $(B,E)$ is also an equal-size pair, where we think of $E$ as an element of $\widehat{\widehat{G}}$ under the identification $\widehat{\widehat{G}} \cong G$.
To show that $T(E,B) = T(B,E)^t$, we observe that the entries $b_j(x_i) = x_i(b_j)$.
Alternatively, we can use Equation \eqref{eq:TEBgeneral} to show that
$$(T(E,B)h,f)_{L^2(E)} = (h, T(B,E)f)_{L^2(B)}$$
for every $f \in L^2(E)$ and $h \in L^2(B)$, where $(\cdot,\cdot)$ is the natural pairing $(f_1,f_2)_{L^2(E)} = \sum_{x\in E} f_1(x)f_2(x)$.

\section{Properties of tightness quantities}
\label{sec:properties}

In what follows, we investigate properties of $L_E(B)$, $U_E(B)$, $\rho_E(B)$, and $D_E(B)$ for an equal-size pair $(E,B)$, and $L(E)$, $U(E)$, $\rho(E)$, and $D(E)$ for a set $E$.

\subsection{Basic properties}
\label{subsec:basicProps}

We begin with immediate properties, whose proofs we omit.

\begin{prop}[Duality]
\label{prop:duality}
$T(E,B)=T(B,E)^t$, $L_E(B)=L_B(E)$, $U_E(B)=U_B(E)$, $\rho_E(B) = \rho_B(E)$, and $D_E(B)=D_B(E)$.
\end{prop}

\begin{prop}[Basis pair]
\label{Basis pair}
Let $(E,B)$ be an equal-size pair.
The following are equivalent:
\begin{itemize}
\item $(E,B)$ is a basis pair;
\item $(B,E)$ is a basis pair; 
\item $T(E,B)$ is invertible;
\item $L_E(B)>0$;
\item $\rho_E(B) < \infty$;
\item $D_E(B) > 0$.
\end{itemize} 
\end{prop}

The next result shows that all  quantities  in Proposition \ref{Basis pair} have nice expressions in terms of the singular values of the matrix  $T(E,B)$. These will prove very useful in understanding their behaviors.  

\begin{prop}
\label{prop:exprsingval}
For an equal-size pair $(E,B)$ with $|E|=n$,  
let  $0 \leq \sigma_1 \leq \sigma_2 \dots\leq \sigma_n$ be the singular values of the matrix $T(E,B)$. Then $L_E(B)=\sigma_1^2$, $U_E(B)=\sigma_n^2$, $\rho_E(B)=\sigma_n^2/\sigma_1^2$, and $D_E(B) = \sigma_1 \sigma_2 \dots \sigma_n$.
Moreover, we have 
$$\sigma_1^2+\sigma_2^2+\dots+\sigma_n^2 = n^2.$$
\end{prop}

\begin{proof}
Put $T:=T(E,B)$. We apply Proposition \ref{prop:exprrieszconst}. Then $\norm{T}=\sigma_n$. If $\sigma_1>0$, then 
  $T$ is invertible and  the singular values of $T^{-1}$ are given by  $1/\sigma_1 \geq 1/\sigma_2 \geq\dots \geq 1/\sigma_n$. Therefore $\norm{T^{-1}} = 1/\sigma_1$. Since $\sigma_1^2,\sigma_2^2,\dots,\sigma_n^2$ are the eigenvalues of $T^*T$,
$$\abs{\det T}^2 = \det T^*T = \sigma_1^2\sigma_2^2\dots\sigma_n^2.$$
Note that all the  diagonal entries of $T^*T$ are equal to $n$, therefore 
$$n^2 = \tr T^*T = \sigma_1^2+\sigma_2^2+\dots+\sigma_n^2.$$
These give all our desired results.
\end{proof}

The following result provides upper and lower  bounds for $L_E(B)$, $U_E(B)$, $\rho_E(B)$, and $D_E(B)$ in terms of  the size of $E$.

\begin{prop}
\label{prop:simplebds}
Let $(E,B)$ be an equal-size pair with $\abs{E}=n$.
Then
$0 \leq L_E(B) \leq n \leq U_E(B) \leq n^2$, $1 \leq \rho_E(B) \leq \infty$, and $0 \leq D_E(B) \leq n^{n/2}$. Moreover, each of these bounds is attained for $E$ and $B$ of arbitrarily large sizes.
\end{prop}
 
\begin{proof} 
The estimations for $L_E(B)$, $U_E(B)$ and $\rho_E(B)$ are straightforward by Proposition \ref{prop:exprsingval}.
For the inequality $D_E(B) \leq n^{n/2}$, we use Proposition \ref{prop:exprsingval}  and {\it the quadratic mean-geometric mean inequality}:
$$D_E(B)^{1/n} = (\sigma_1\sigma_2 \dots\sigma_n)^{1/n}
\leq \sqrt{\frac{\sigma_1^2+\sigma_2^2+\dots+\sigma_n^2}{n}} = \sqrt{n}.$$
Example \ref{ex:simplebdsattained} below shows that these bounds can be attained for $E$ and $B$ of arbitrarily large sizes.
\end{proof}

\begin{example}
\label{ex:simplebdsattained}
Let  $G=\mb Z_m^2$, and take 
$$E=\set{(0,0),(1,0),\dots,(m-1,0)} \quad \text{and} \quad B=\set{(0,0),(0,1),\dots,(0,m-1)}.$$
It is easily verified that for this pair 
$L_E(B)=0$, $U_E(B)=n^2$, $\rho_E(B)=\infty$ and $D_E(B)=0$.
If $E=B=\Z_m^d$, we obtain $L_E(B)=U_E(B)=n$, $\rho_E(B)=1$  and $D_E(B) = n^{n/2}$.
To show this, notice  $T(E,B)/\sqrt{n}$ is a unitary matrix, so all the singular values of $T(E,B)$ are  equal to $\sqrt{n}$.
\end{example}

\begin{remark}
\label{remark:HadamardIneq}
The result that $D_E(B) \leq n^{n/2}$ in Proposition \ref{prop:simplebds} can be given an alternative proof using {\it  Hadamard's inequality}. This inequality states  that for a given matrix $M$ with    $n$ columns $\{v_i\}_{i=1}^n$,
$$|\det M|\leq \Pi_{i=1}^n \|v_i\| .$$
Applying Hadamard's inequality to the  $n\times n$  matrix $T$, and noting that $\|v_i\|^2=n$, 
 we obtain    $\abs{\det T} \leq n^{n/2}$. The traditional proof of Hadamard's inequality is similar to Proposition \ref{prop:simplebds}'s proof.
\end{remark}
 
 \vskip.124in 

We can now characterize a spectral pair in terms of our defined quantities.

\begin{prop}[Spectral pair] 
\label{prop:basicspectral} 
For a given equal-size pair $(E,B)$, 
the following are equivalent:
\begin{itemize}
\item $(E,B)$ is a spectral pair;
\item $(B,E)$ is a spectral pair;
\item $T(E,B)/\sqrt{n}$ is a unitary matrix; 
\item $L_E(B)=U_E(B)$;
\item $L_E(B)=n$;
\item $U_E(B)=n$;
\item $\rho_E(B)=1$;
\item $D_E(B) = n^{n/2}$.
\end{itemize}
\end{prop}

\begin{proof}
By Proposition \ref{prop:exprsingval}, the last five conditions are all equivalent to the statement that all singular values of $T(E,B)$ are the same and are equal  to $\sqrt{n}$. This, in turn, is equivalent to $T(E,B)/\sqrt{n}$ being a unitary matrix. By definition, this is equivalent to $(E,B)$ being a spectral pair. From this, we obtain duality between $E$ and $B$.
\end{proof}

From Propositions \ref{prop:simplebds} and \ref{prop:basicspectral}, notice that a spectral pair $(E,B)$ has the largest $L_E(B)$, smallest $U_E(B)$, smallest $\rho_E(B)$, and largest $D_E(B)$ among all pairs of the same sizes.
Thus, we can think of a pair $(E,B)$ with large $L_E(B)$, small $U_E(B)$, small $\rho_E(B)$, and large $D_E(B)$ as being ``close to spectral.''

From Definition \ref{def:setquantities}, we can now see that $L(E)$, $U(E)$, $\rho(E)$, and $D(E)$ are defined to be the corresponding quantities for a pair $(E,B)$ that are closest to being spectral with respect to these quantities.

Recall that the definition of a spectral pair means that $\chi(\inner{\cdot,b}/m)$, $b \in B$, are orthogonal on $E$. The following proposition makes the intuition of ``being close to spectral'' precise.

\begin{prop}
\label{prop:closeortho}
Fix $n$, and let $(E,B)$ be an equal-size pair with $\abs{E}=n$.
Any of the following statements implies all others:
\begin{itemize}
\item The exponential functions $\chi(\inner{\cdot,b}/m)$, $b \in B$, are close to being orthogonal on $E$;
\item $L_E(B)$ is close to $n$;
\item $U_E(B)$ is close to $n$;
\item $\rho_E(B)$ is close to $1$;
\item $D_E(B)$ is close to $n^{n/2}$.
\end{itemize}
More precisely, define the angle $\theta \in [0,\pi/2]$ between two nonzero vectors $c,d \in \C^n$ by $$\cos \theta = \frac{\abs{ \sum_{i=1}^n c_i \overline{d_i}}}{\norm{c}\norm{d}}.$$
For a matrix $A$ with nonzero columns, define $\operatorname{ortho}(A)$ to be the maximum of $\pi/2-\theta$ over all angles $\theta$ between two distinct columns of $A$.
Then, if any of the following quantities is close to zero, then all others must also be close to zero: $\operatorname{ortho}(T(E,B))$, $L_E(B)-n$, $U_E(B)-n$, $\rho_E(B)-1$, and $D_E(B) - n^{n/2}$.
\end{prop}

\begin{proof}
Let $T=T(E,B)$.
Proposition \ref{prop:exprsingval} implies that each of $L_E(B)-n$, $U_E(B)-n$, $\rho_E(B)-1$, and $D_E(B) - n^{n/2}$ is close to zero exactly when all singular values of $T(E,B)$ are close to $\sqrt{n}$.
This proves the equivalence between these statements.

Now if $\operatorname{ortho}(T)$ is close to zero, then $T^*T$ is close to $nI$. By continuity, all singular values of $T$ are close to $\sqrt{n}$. 
Conversely, suppose that $\operatorname{ortho}(T)$ is not close to zero. Specifically, assume that the angle $\theta$ between columns $i$ and $j$ is not close to $\pi/2$.
Let $e_1,e_2,\dots,e_n$ be the standard basis of $\C^n$.
Pick $\omega \in \C$ of norm one such that
$$\omega \sum_{k=1}^n (Te_i)_k\overline{(Te_j)_k} = \cos\theta \norm{Te_i} \norm{Te_j},$$
that is, $\omega$ should ``dephase'' the inner product.
Then
$$\norm{T(\omega e_i+ e_j)}^2 = \norm{Te_i}^2+\norm{Te_j}^2 + 2 \cos\theta \norm{Te_i}\norm{Te_j}
= 2n + 2n\cos\theta,$$
while $\norm{\omega e_i+e_j} = \sqrt{2}$.
So $\norm{T} \geq \sqrt{n(1+\cos\theta)}$.
Because $\cos\theta$ is not close to zero, $\norm{T}$ is not close to $\sqrt{n}$, so $U_E(B)$ is not close to $n$.
\end{proof}

We can now put bounds on $L(E)$, $U(E)$, $\rho(E)$, and $D(E)$ and characterize spectral sets in a similar way.

\begin{prop}
Let $\abs{E}=n$.
Then $0 < L(E) \leq n$, $n \leq U(E) < n^2$ if $n>1$, $\sqrt{U(E)/L(E)} \leq \rho(E) < \infty$, and $0 < D(E) \leq n^{n/2}$. 
\end{prop}

\begin{proof}
Straightforward by the definitions and Proposition \ref{prop:simplebds}.
The strict inequalities follow from the fact that for every $E$, there is a $B$ such that $(E,B)$ is a basis pair (Cor. \ref{cor:basispairexist}).
\end{proof}

The proposition below shows that a spectral set has the largest $L(E)$, smallest $U(E)$, smallest $\rho(E)$, and largest $D(E)$ among all sets of the same sizes. We omit its proof.

\begin{prop}[Spectral set]
\label{prop:basicPropSpec}
Let $\abs{E}=n$.
The following are equivalent:
\begin{itemize}[noitemsep]
\item $E$ is spectral;
\item $L(E)=U(E)$;
\item $L(E) = n$;
\item $U(E) = n$;
\item $\rho(E) = 1$;
\item $D(E) = n^{n/2}$.
\end{itemize}
\end{prop}

\subsection{Invariance properties}
\label{subsec:invarianceProps}

Our defined quantities also have invariance properties which we can exploit. We begin with translational invariance.

\begin{prop}[Translational invariance] 
\label{prop:invtrans}
\
\begin{itemize}
\item $L_E(B)$, $U_E(B)$, $\rho_E(B)$, and $D_E(B)$ are invariant under translations of $E$ and $B$.
\item $L(E)$, $U(E)$, $\rho(E)$, and $D(E)$ are invariant under translations of $E$.
\end{itemize}
\end{prop}

\begin{proof}
Translation of $E$ and $B$ changes $T=T(E,B)$ to $D_1 T D_2$, where $D_i$ are diagonal
with diagonal entries of absolute value 1.
Specifically, if $E=\set{e_1,\dots,e_n}$ and $B=\set{b_1,\dots,b_n}$,
$$T(E+x,B+y)=y(x)\, \text{diag}[y(e_1),\dots,y(e_n)]\,T(E,B)\,\text{diag}[b_1(x),\dots,b_n(x)].$$
Notice that the $D_i$ are unitary, so the singular values of $D_1 T D_2$ and $T$ are equal.
By Proposition \ref{prop:exprsingval}, the results for pairs $(E,B)$ follow, and these imply the corresponding results for sets $E$.
\end{proof}

The rest of the invariance results are best stated in the setting of a finite abelian group $G$ (see Sec. \ref{subsec:prelimGeneralFinite}), so we will do so.
First, Proposition \ref{prop:invtrans} holds in this setting with the same proof.

Recall that a group homomorphism in this setting is called a $\Z$-linear transformation. For $G=\Z_p^d$, $p$ a prime, this is a linear transformation in the sense of a vector space.

\begin{prop}
\label{prop:invlinear}
Let $E \subseteq G$. The quantities $L(E)$, $U(E)$, $\rho(E)$, and $D(E)$ are invariant under invertible $\Z$-linear transformations.
\end{prop}

\begin{proof}
Follows because these quantities can be defined without choosing generators for $G$.
\end{proof}

In fact, we can say more.
Notice that a $\Z$-linear transformation $A: G \to G$ induces the transpose $A^t: \widehat{G} \to \widehat{G}$ defined by $A^t(b) = b \circ A$.
For $G=\Z_p^d$, $p$ a prime, this is the usual transpose of a linear transformation. If $A$ is invertible, we can check that $A^t$ is invertible with $(A^t)^{-1} = (A^{-1})^t$; denote this by $A^{-t}$.
We have the following.

\begin{prop}
\label{prop:invlineareb}
Let $(E,B)$ be an equal-size pair. For an invertible $\Z$-linear transformation $A:G \to G$, let $(E',B')=(AE,A^{-t}B)$.
Then $L_{E'}(B') = L_E(B)$, $U_{E'}(B') = U_E(B)$, $\rho_{E'}(B')=\rho_E(B)$, and $D_{E'}(B') = D_E(B)$.
\end{prop}

\begin{proof}
The entries of the Fourier matrices are the same:
$$A^{-t}b_j(Ax_i) = b_j\circ A^{-1}(Ax_i) = b_j(x_i),$$
so $T(E,B)$ and $T(E',B')$ are essentially the same linear transformation.

Alternatively, we can use Equation \eqref{eq:TEBgeneral} to check this fact in a coordinate-free way. Specifically, we can show that for every $h \in L^2(B')$,
$$T(E,B)(h \circ A^{-t}) \circ A^{-1} = T(E',B')h.$$
\end{proof}

Recall that an affine transformation is the composition of a $\Z$-linear transformation followed by a translation. 

\begin{cor}
\label{cor:invaffine}
Let $E \subseteq G$. The quantities $L(E)$, $U(E)$, $\rho(E)$, and $D(E)$ are invariant under invertible affine transformations.
\end{cor}

\begin{proof}
Follows from Propositions \ref{prop:invtrans} and \ref{prop:invlinear}.
\end{proof}

Recall also the notation $L(E;G)$, etc., which means $L(E)$, etc., when the group $G$ needs to be specified.

\begin{prop}[Affine restriction]
\label{prop:affinerestrict}
Suppose that $H$ is a direct summand of $G$, that is, there is a subgroup $K \subseteq G$ such that $G=H \oplus K$. Let $E$ be contained in the coset $H+a$, thought of as a group with the group structure inherited from $H$. Then
$L(E;H+a)=L(E;G)$, $U(E;H+a)=U(E;G)$, $\rho(E;H+a)=\rho(E;G)$, and $D(E;H+a)=D(E;G)$.
\end{prop}

\begin{proof}
Let $Q$ denote any of the quantities $-L$, $U$, $\rho$, and $-D$.
By Proposition \ref{prop:invtrans}, we can assume that $a=0$.
Indeed, if $E=E'+a$ where $E' \subseteq H$, then
$Q(E;H+a)=Q(E'+a;H+a)=Q(E';H)$ and $Q(E;G)=Q(E'+a;G)=Q(E';G)$.
So suppose that $a=0$.
Our proof is in the notation of dual groups.

Recall that $\widehat{G}\cong\widehat{H}\times\widehat{K}$, canonically, with the following identification.
For $\hat{h}\in\widehat{H}$ and $\hat{k}\in\widehat{K}$,
let $(\hat{h},\hat{k})$ act as an element of $\widehat{G}$ by
$(\hat{h},\hat{k})(h,k) = \hat{h}(h)\hat{k}(k)$
for $h \in H$ and $k \in K$.
We define the projection $p((\hat{h},\hat{k}))=(\hat{h},1)$, where ``$1$" denotes the trivial homomorphism.

Let $B\subseteq\widehat{G}$ be such that $(E,B)$ is a basis pair.

\emph{Claim 1}. $|p(B)|=|B|$.

\emph{Proof of Claim 1}. Clearly $|p(B)|\leq |B|$. Hence, if $|p(B)|\neq |B|$, there exist $\hat{g}_{1}, \hat{g}_{2}\in B$, $\hat{g}_{1}\neq\hat{g}_{2}$, with $p(\hat{g}_{1})=p(\hat{g}_{2})$. Set $(\hat{h}_{1}, 1)=p(\hat{g}_{1})$ and $(\hat{h}_{2}, 1)=p(\hat{g}_{2})$. Given $x\in E$, as $x=x+0$ with $x\in H$ and $0\in K$, we have, for $(\hat{h}_{1},\hat{k}_{1})=\hat{g}_{1}$, that $\hat{g}_{1}(x)=(\hat{h}_{1}(x), \hat{k}_{1}(0))$. Hence, $\hat{g}_{1}(x)=(\hat{h}_{1}(x),1)=p(\hat{g}_{1})(x)$. Similarly, we have $p(\hat{g}_{2})(x)=(\hat{h}_{2}(x),1)=\hat{g}_{2}(x)$. As $p(\hat{g}_{1})=p(\hat{g}_{2})$, it follows that $\hat{g}_{1}(x)=\hat{g}_{2}(x)$ for all $x\in E$, so $\{E\ni x\mapsto\hat{g}(x)\}_{\hat{g}\in B}$ is linearly dependent, a contradiction with $(E, B)$ being a basis pair. Thus, $|p(B)|=|B|$.\ \qedsymbol

\emph{Claim 2}. $T(E,B)$ and $T(E,p(B))$ are equal up to a permutation of the columns.

\emph{Proof of Claim 2}.
From the proof of Claim 1, we have the equality of matrices $[\hat{g}(x)]_{x\in E, \hat{g}\in B}=[p(\hat{g})(x)]_{x\in E, p(\hat{g})\in p(B)}$, up to a choice of ordering.\ \qedsymbol

\emph{Claim 3}. $(E,p(B))$ is a basis pair, and $Q_E(B)=Q_E(p(B))$.

\emph{Proof of Claim 3}.
Follows directly from Claim 2.\ \qedsymbol

\emph{Claim 4}. We have $Q(E;H)=Q(E;G)$.

\emph{Proof of Claim 4}. First, as we can identify $\widehat{H}$ with $\widehat{H}\times\{1\}\subseteq\widehat{G}$, we see that
\[
Q(E;H)=\min_{B\subseteq\widehat{H}}Q_{E}(B)=
\min_{B\subseteq\widehat{H}\times\{1\}}Q_{E}(B)\geq \min_{B\subseteq\widehat{G}}Q_{E}(B)=Q(E;G).
\]
Moreover, given $B\subseteq\widehat{G}$ where $(E,B)$ is a basis pair, we have $p(B)\subseteq\widehat{H}\times\{1\}$, so $Q_{E}(B)=Q_{E}(p(B))\geq\min_{B\subseteq\widehat{H}\times\{1\}}Q_{E}(B)=Q(E;H)$, whence $Q(E;G)=\min_{B\subseteq\widehat{G}}Q_{E}(B)\geq Q(E;H)$.
We conclude that $Q(E;H)=Q(E;G)$.\ \qedsymbol
\end{proof}

Proposition \ref{prop:affinerestrict} has two important corollaries.
The first corollary is for the case $G=\Z_p^d$, $p$ a prime, where $G$ has the extra structure of a vector space.

\begin{cor}
\label{cor:affinerestrictZpd}
Let $V$ be a vector space over $\Z_p$, $W$ an affine subspace of $V$ viewed as a vector space where we take any element to be the zero element, and $E \subseteq W$.
Then $L(E;W)=L(E;V)$, $U(E;W)=U(E;V)$, $\rho(E;W)=\rho(E;V)$, and $D(E;W)=D(E;V)$.
\end{cor}

For $G=\Z_{n_1}\times\Z_{n_2}\times\dots\times\Z_{n_k}$ and $S \subseteq \set{1,2,\dots,k}$, define the \emph{coordinate subspace of $G$ with respect to $S$} to consist of elements $x$ of $G$ where for any $i \in S$, the $i$th coordinate of $x$ is zero. A \emph{coordinate affine subspace} of $G$ is a translate of some coordinate subspace of $G$.

The second corollary of Proposition \ref{prop:affinerestrict} is the following. It is actually equivalent to Proposition \ref{prop:affinerestrict} because we can pick generators for $H$ and $K$ in that proposition.

\begin{cor}
Let $G=\Z_{n_1}\times\Z_{n_2}\times\dots\times\Z_{n_k}$, $H$ be a coordinate affine subspace of $G$ viewed as a group where we take any element to be the zero element, and $E \subseteq H$.
Then $L(E;H)=L(E;G)$, $U(E;H)=U(E;G)$, $\rho(E;H)=\rho(E;G)$, and $D(E;H)=D(E;G)$.
\end{cor}

From the proof of Proposition \ref{prop:affinerestrict}, we also obtain the following.
If $H$ is a direct summand of $G$ and $E \subseteq H$, in order for $(E,B)$ to be a basis pair, no two elements of $B$ can act on $H$ in the same way.
We can generalize this statement somewhat to the case where $E$ is ``almost'' contained in $H$, as shown below.

\begin{prop}
\label{prop:almostinsubgp}
Let $H$ be a direct summand of $G$. Define an equivalence relation $\sim$ on $\widehat{G}$ as follows: $b_1 \sim b_2$ if the restrictions of $b_1$ and $b_2$ to $H$ are equal.
Let $E \subseteq G$ be such that $\ell$ elements of $E$ are not contained in $H$, and let $(E,B)$ is a basis pair.
For each equivalence class $C \in \widehat{G}/\sim$,
let $\operatorname{cnt}(C)$ be the number of elements of $B$ in $C$. Then
$$\sum_{C \in \widehat{G}/\sim, \operatorname{cnt}(C) \geq 1} (\operatorname{cnt}(C)-1) \leq \ell.$$
\end{prop}

\begin{proof}
Suppose that the sum in question is greater than $\ell$, and we show that $T(E,B)$ must be singular. Let $E=\set{x_1,x_2,\dots,x_n}$, where $x_1,x_2,\dots,x_{n-\ell} \in H$. View $T(E,B)$ as a matrix whose row $i$ corresponds to $x_i$.

Consider a $C \in \widehat{G}/\sim$ with $\operatorname{cnt}(C) \geq 1$. Let $c = \operatorname{cnt}(C)$
and let $b_1,\dots,b_c \in B \cap C$.
Then the columns corresponding to $b_1,\dots,b_c$ all have the same values in the first $n-\ell$ rows.
If we subtract the column corresponding to $b_1$ from the columns corresponding to $b_2,\dots,b_c$, we obtain $c-1$ columns whose first $n-\ell$ entries are zero, and the resulting matrix is singular if and only if $T(E,B)$ is  singular.
Now repeat this process for all $C \in \widehat{G}/\sim$ with $\operatorname{cnt}(C)\geq 1$ to obtain a matrix whose $\sum_{C \in \widehat{G}/\sim,\operatorname{cnt}(C)\geq 1} (\operatorname{cnt}(C)-1) > \ell$ columns have zeros in their first $n-\ell$ entries. Since the nonzero entries of these columns can only be in the last $\ell$ entries, the columns must be linearly dependent. Therefore, $T(E,B)$ is singular.
\end{proof}

\begin{example}
\label{ex:condtoinfty}
We exhibit a family of subsets of $\mathbb{Z}_m^{2}$, one for each $m \geq 2$, whose Riesz ratios tend to infinity as $m\to\infty$.
Let $G=\mathbb{Z}_m^{2}$, let $H=\mathbb{Z}_m\times\{0\}$, and let $K=\{0\}\times\mathbb{Z}_m$, so $G=H\oplus K$.
Let $E=H\cup \{(0,1)\}$.
See Figure \ref{fig:rieszRatioToInfty}.

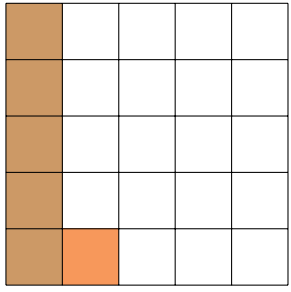
\begin{figure}
\begin{center}
\begin{tikzpicture}[scale=0.75]
\fill[brown!80!white] (0,0) rectangle (1,5);
\fill[Orange!80!white] (1,0) rectangle (2,1);
\draw (0,0) grid (5,5);
\end{tikzpicture}
\end{center}
\caption{A set in the family in Example \ref{ex:condtoinfty} whose Riesz ratios tend to infinity.}
\label{fig:rieszRatioToInfty}
\end{figure}

Let $B$ be such that $(E,B)$ is a basis pair.
We first show that, under $\sim$ from Proposition \ref{prop:almostinsubgp}, exactly one equivalence class contains two elements from $B$, while any other equivalence class contains one element from $B$.
We have $\widehat{G}/\sim=\cup_{t=0}^{m-1}[(t,0)]$.
Notice that $\ell=1$ element of $E$ is not contained in $H$. 
By Proposition \ref{prop:almostinsubgp}, we also have $\sum_{t=0\, :\, |B\cap [(t,0)]|\geq 1}^{m-1}\left(|B\cap [(t,0)]|-1\right)\leq 1$. If $|B\cap [(t,0)]|=0$ for some $t$, then since $|E|=m+1$, the pigeonhole principle implies that the above inequality is violated. So, for all $0\leq t\leq m-1$, we have $|B\cap [(t,0)]|\geq 1$. Hence, $|B\cap [(s,0)]|=2$ for a unique $s$ with $0\leq s\leq m-1$, and $|B\cap [(t,0)]|=1$ for any other $0 \leq t\leq m-1$.

Write $B=\{b_{1}, b_{2}, \dots, b_{m+1}\}$, with $b_{1}, b_{2}$ in $[(s,0)]$. Then, $b_{1}$ and $b_{2}$ agree on $H$. Writing $E$ as $E=\{x_{1}, \dots, x_m, x_{m+1}\}$ with $x_{i}=(i-1,0)$ for $1\leq i\leq m$ and $x_{m+1}=(0,1)$, we see that the vectors $(b_{1}(x_{i}))_{i=1}^{m+1}$ and $(b_{2}(x_{i}))_{i=1}^{m+1}$ differ only in their last entry. So, $\langle b_{1}, b_{2}\rangle=m+z$, for some $z\in\mathbb{C}$ with $|z|=1$. In particular, $|\langle b_{1}, b_{2}\rangle|\geq m-1$. We also observe that $\|(b_{1}(x_{i}))_{i=1}^{m+1}\|=\sqrt{m+1}=\|(b_{2}(x_{i}))_{i=1}^{m+1}\|$. Recall that
\[
L_{E}(B)=\inf_{(c_{i})\neq 0}\frac{\|\sum_{i=1}^{m+1}c_{i}b_{i}\|_{L^{2}(E)}^{2}}{\sum_{i=1}^{m+1}|c_{i}|^{2}}.
\]
Writing $\langle b_{1}, b_{2}\rangle = re^{i\theta}$ with $r\geq 0$ and $\theta\in\mathbb{R}$, we take $(c_{i})_{i=1}^{m+1}$ in $\mathbb{C}^{m+1}$ such that $c_{1}=-e^{-i\theta}$, $c_{2}=1$, and $c_{i}=0$ for $i>2$. Then we have
\[
L_{E}(B)\leq \frac{\|c_{1}b_{1}+c_{2}b_{2}\|_{L^{2}(E)}^{2}}{|c_{1}|^{2}+|c_{2}|^{2}}.
\]
Writing
\[
\frac{\|c_{1}b_{1}+c_{2}b_{2}\|_{L^{2}(E)}^{2}}{|c_{1}|^{2}+|c_{2}|^{2}}=\frac{\langle c_1 b_1 + c_2 b_2 , c_1 b_1 + c_2 b_2\rangle}{2},
\]
we obtain
\[
L_{E}(B)\leq \frac{\|(b_{1}(x_{i}))_{i=1}^{m+1}\|_{2}^{2}+\|(b_{2}(x_{i}))_{i=1}^{m+1}\|_{2}^{2}+2\text{Re}(\langle c_{1}b_{1}, c_{2}b_{2}\rangle)}{2},
\]
so
\[
L_{E}(B)\leq\frac{2(m+1)-2|\langle b_{1}, b_{2}\rangle|}{2}\leq (m+1)-(m-1)=2.
\]
Thus, $L(E)\leq 2$. As $m+1\leq U(E)$, it follows that $\rho(E) \geq (m+1)/2 \to\infty$ as $m\to\infty$.
\end{example}

\begin{remark}
In Example \ref{ex:condtoinfty}, we actually need not use
Proposition \ref{prop:almostinsubgp}.
Indeed, by the Pigeonhole principle, there are distinct $b_1,b_2 \in B$ that are in the same equivalence class.
Then the argument can proceed as in the last paragraph.
\end{remark}

The corollary below is Proposition \ref{prop:almostinsubgp} but with generators of $G$ already chosen.

\begin{cor}
Let $G=\Z_{n_1}\times\Z_{n_2}\times\dots\times\Z_{n_k}$, $S$ and $S'$ be a partition of $\set{1,2,\dots,k}$, and $H$ and $K$ be the coordinate subspaces of $G$ with respect to $S$ and $S'$, respectively. Let $E \subseteq G$ be such that $\ell$ elements of $E$ are not contained in $H$, and let $(E,B)$ be a basis pair. Then
$$\sum_{x \in H, \abs{B\cap (K+x)} \geq 1} (\abs{B \cap (K+x)} - 1) \leq \ell.$$
\end{cor}

We end with results on scale invariance. Consider a $k \in \Z$ such that multiplication by $k$ is an invertible $\Z$-linear transformation.
For example, if $G=\Z_p^d$, $p$ a prime, then any $1 \leq k \leq p-1$ will do.
By \cite[Cor. 4.3]{REU}, we know that in $\Z_p^d$, if $(E,B)$ is a spectral pair, then $(E,kB)$ is also a spectral pair. This statement generalizes to any finite abelian group $G$ as shown below.

\begin{prop}
\label{prop:scalespec}
Let $k \in \Z$ be such that multiplication by $k$ is an invertible $\Z$-linear transformation on $G$.
If $(E,B)$ is a spectral pair, then $(E,kB)$ is also a spectral pair.
\end{prop}

\begin{proof}
View $G \cong \Z_{n_1}\times\Z_{n_2}\times\dots\times\Z_{n_\ell}$ and identify $\widehat{G}$ with $G$ as in Equation \eqref{eq:ghatacts}. Then $k$ must be relatively prime to $n_1,n_2,\dots,n_\ell$.
Let $\omega = \chi(1/(n_1n_2\dots n_\ell))$ be a primitive $n_1n_2\dots n_\ell$-th root of unity. Then the entries of the matrix $T(E,B)$ are powers of $\omega$. Specifically, let $E=\set{x_1,x_2,\dots,x_n}$ and $B=\set{b_1,b_2,\dots,b_n}$. If $M(x)$ is the $n \times n$ matrix whose $(i,j)$ entry is
\begin{equation}
\label{eq:entryMx}
x^{\paren{\frac{(x_i)_1 (b_j)_1}{n_1}+\frac{(x_i)_2 (b_j)_2}{n_2}+\dots+\frac{(x_i)_\ell (b_j)_\ell}{n_\ell}}n_1n_2\dots n_\ell},
\end{equation}
then $T(E,B)=M(\omega)$.

Notice that $T(E,kB)=M(\omega^k)$.
For $x \in \mb{S}^1$, let $M_i(x)$ be column $i$ of $M(x)$, and let $P_{ij}(x) = M_i(x)^T M_j(x^{-1}) \in \Z[x,x^{-1}]$ be the inner product between columns $i$ and $j$ of $M(x)$.
The fact that $(E,B)$ is a spectral pair means that $P_{ij}(\omega) = 0$ for all $i \neq j$.
Because $k$ is relatively prime to $n_1n_2\dots n_\ell$, this implies that $P_{ij}(\omega^k)=0$ for all $i \neq j$.
Hence, the columns of $T(E,kB)$ are orthogonal, and so $(E,kB)$ is a spectral pair.
\end{proof}

By Proposition \ref{prop:scalespec}, if $(E,B)$ is a spectral pair, then our defined quantities on $(E,B)$ are conserved when $B$ is multiplied by $k$.
However, this result does not generalize beyond spectral pairs, as the following example shows.

\begin{example}
Let $m>2$.
Consider $E=B=\set{0,1}$ in $\Z_m$, and let $k$ be relatively prime to $m$. Then $(E,B)$ is a non-spectral pair. We have
$$T(E,B) = \begin{pmatrix} 1 & 1 \\ 1 & \chi(1/m) \end{pmatrix}, \quad
T(E,kB) = \begin{pmatrix} 1 & 1 \\ 1 & \chi(k/m) \end{pmatrix},$$
so none of the quantities $L_E(B)$, $U_E(B)$, $\rho_E(B)$, and $D_E(B)$ is conserved.
\end{example}

\section{Bounds on tightness quantities}
\label{sec:bounds}

In this section, we obtain bounds on how far a basis pair $(E,B)$ and a set $E$ can be from being spectral. The quantities we will use to measure this are $\rho_E(B)$ and $\rho(E)$, where these quantities being close to one and $\infty$ translate to $(E,B)$ or $E$ being close to and far from being spectral, respectively.

Our main findings are the following.

\begin{itemize}
\item In the special case of $\abs{E}=2$ and $E \subseteq \Z_p^d$, $p$ a prime, as $p \to \infty$, $\rho(E) \to 1$ independent of $d$.
However, this result cannot be generalized beyond $G=\Z_p^d$ or beyond $\abs{E}=2$. (See Sec. \ref{subsec:boundsNearlySpec}.)
\item For $G\cong \Z_{n_1}\times\Z_{n_2}\times\dots\times\Z_{n_\ell}$, let $M=\lcm(n_1,n_2,\dots,n_\ell)$.
Then for $E\subseteq G$, $\rho_E(B)$ is bounded if $\abs{E}=n$ is fixed and $M$ is bounded.
In particular, for $E\subseteq \Z_m^d$, $\rho_E(B)$ is bounded for $\abs{E}=n$ and $m$ fixed, independent of $d$.
Moreover, the condition that $M$ is bounded cannot be removed. (See Sec. \ref{subsec:boundsUpperEB}.)
\item Surprisingly, $\rho(E)$ is bounded for fixed $\abs{E}=n$, completely independent of $G$.
This shows that in those cases where results of Section \ref{subsec:boundsNearlySpec} do not apply and we do not have $\rho(E) \to 1$, $\rho(E)$ is still always bounded.
(See Sec. \ref{subsec:boundsUpperE}.)
\end{itemize}

In summary, if one wants to construct a sequence of $(E,B)$ such that $\rho_E(B) \to \infty$, one cannot fix $\abs{E}$ and $M$.
If one wants to construct a sequence of $E$ such that $\rho(E) \to \infty$, one cannot fix $\abs{E}$.

\subsection{Nearly spectral sets of fixed size}
\label{subsec:boundsNearlySpec}

In this subsection, we investigate the following question.
If we fix the size of the set $E$ and let the ambient space $\Z_m^d$ ``grow to infinity,'' does $E$ always become closer and closer to being spectral?
Recall from Proposition \ref{prop:closeortho} that the statement that a set $E$ is close to being spectral can be formalized as $\rho(E)$ being close to one.

Proposition \ref{prop:size2closespec} shows this to be the case for any sequence $E_p \subseteq \Z_p^{d_p}$, $p$ being primes, with $\abs{E_p}=2$.
We might interpret this fact as follows.
As $p \to \infty$, there is more space, so it is becoming easier to find a partner $B_p$ such that $(E_p,B_p)$ is close to being spectral.

Nevertheless, we also show in Examples \ref{ex:Zmsize2notclosespec} and \ref{ex:Zpsize3notclosespec} that this result does not hold when the size of $E_p$ changes to 3 or when we consider a sequence $E_m \subseteq \Z_m^{d_m}$, $m \geq 1$ is not necessarily prime.
Thus, the phenomenon of getting close to being spectral seems to be special
to the case of sets of size two in $\Z_p^d$, $p$ a prime.

Later, in Section \ref{subsec:boundsUpperE}, we show that the above intuition of having more space still holds in some weaker sense.
Namely, for $E$ of fixed size, $\rho(E)$ is \emph{bounded} independent of the ambient space $G$, where $G$ is any finite abelian group.
So $E$ may not get closer and closer to being spectral as $G$ grows, but it also cannot get arbitrarily far away.

\begin{prop}
\label{prop:size2closespec}
For any sequence of sets $E_p \subseteq \Z_p^{d_p}$ indexed by primes $p$ with $\abs{E_p} = 2$, we have
$\rho(E_p)\to 1$ as $p \to \infty$.
\end{prop}

\begin{proof}
By Corollary \ref{cor:affinerestrictZpd}, we can assume that $E_p \subseteq \Z_p$.
By a further affine transformation (Cor. \ref{cor:invaffine}), we can assume that
$E_p=\set{0,1}$.

For each $p\geq 3$, choose $B_p=\set{0,(p-1)/2}$. Then
$$\rho_{E_p}(B_p) = \cond^2 \begin{pmatrix} 1 & 1 \\
1 & \chi((p-1)/{2p}) \end{pmatrix}.$$
By the continuity of the condition number,
$$\rho_{E_p}(B_p) \to \cond^2 \begin{pmatrix} 1 & 1 \\ 1 & -1 \end{pmatrix} = 1,$$
as $p \to \infty$.
Therefore, $\rho(E_p) \to 1$ as $p \to \infty$.
\end{proof}

As mentioned at the start of the subsection, Proposition \ref{prop:size2closespec} cannot be generalized beyond $G=\Z_p^d$, $p$ a prime, or beyond sets of size two, as the following two examples demonstrate.

\begin{example}
\label{ex:Zmsize2notclosespec}
We present a sequence $E_m \subseteq \Z_m^{d_m}$, $m \geq 1$, with $\abs{E_m}=2$, such that $\rho(E_m)$ does not tend to 1 as $m \to \infty$.

Take $E_m = \set{0,m/3} \subseteq \Z_m$ for $m$ that are multiples of 3, and let $B_m$ be such that $\rho_{E_m}(B_m)$ is minimized.
By translational invariance (Prop. \ref{prop:invtrans}), we can assume that $0 \in B_m$. Then
$$\rho_{E_m}(B_m) = \cond^2 \begin{pmatrix} 1 & 1 \\ 1 & \chi(k_m/3) \end{pmatrix}$$
for some $k_m \in \set{0,1,2}$. None of the $k_m$ makes the right-hand side one, so $\rho_{E_m}(B_m)$ cannot converge to one.
\end{example}

\begin{example}
\label{ex:Zpsize3notclosespec}
We present a sequence $E_p \subseteq \Z_p^{d_p}$, $p$ being primes, with $\abs{E_p}=3$, such that $\rho(E_p)$ does not tend to 1 as $p \to \infty$.

Take $E_p = \set{0,1,3} \subseteq \Z_p$, and let $B_p$ be such that $\rho_{E_p}(B_p)$ is minimized. Again we can assume that $0 \in B_p$.
Suppose to the contrary that $\rho_{E_p}(B_p) \to 1$ as $p \to \infty$. Notice that
$$T(E_p,B_p)= \begin{pmatrix} 1 & 1 & 1 \\ 1 & x_p & y_p \\ 1 & x_p^3 & y_p^3 \end{pmatrix},$$
for some $x_p, y_p \in \mb{S}^1$.
For large $p$, since $\rho_{E_p}(B_p)$ is close to 1, Proposition \ref{prop:closeortho} implies that the columns of $T(E,B)$ are close to being orthogonal. Hence $x_p+x_p^3$ is close to $-1$. It is easily verified that the sum of two elements $a,b \in \mb{S}^1$ is close to $-1$ if and only if $(a,b)$ is close to $(\chi(1/3),\chi(2/3))$ or $(\chi(2/3),\chi(1/3))$.
Because $(x_p,x_p^3)$ cannot be close these pairs, we obtain a contradiction.
So $\rho_{E_p}(B_p)$ cannot tend to 1 as $p \to \infty$.
\end{example}

\subsection{Relations between tightness quantities}
\label{subsec:boundsRelations}

In this section, we introduce quantitative bounds that relate the tightness quantities to one another. Our main tool is Proposition \ref{prop:exprsingval}.
Recall from Proposition \ref{prop:closeortho} that, for a fixed $\abs{E}=n$, any of the following statements implies all others: $L_E(B)$ is close to $n$, $U_E(B)$ is close to $n$, $\rho_E(B)$ is close to 1, and $D_E(B)$ is close to $n^{n/2}$.
Thus, results in this section can be thought of as a more precise version of that proposition.

Apart from obtaining bounds when $(E,B)$ is close to being spectral, we will also use these results to derive estimates on how far $(E,B)$ and $E$ can be from being spectral in Sections \ref{subsec:boundsUpperEB} and \ref{subsec:boundsUpperE}.
Specifically, we first bound $D_E(B)$ away from zero, this task being the easiest since $D_E(B)$ is the absolute value of an integer linear combination of roots of unity.
Then, estimates on other tightness quantities will follow.

Denote $L_E(B)$ by $L$, etc. We first describe how $L$, $U$, and $C$ can be related to one another. Observe that
$$(n-1)L+U = (n-1)\sigma_1^2+\sigma_n^2
\leq \sigma_1^2+\sigma_2^2+\dots+\sigma_n^2=n^2.$$
Similarly, $L+(n-1)U \geq n^2$.
From these two inequalities, we can deduce a lower and upper bound of $L$ in terms of $U$, and vice versa.
Moreover, these inequalities become equalities when $\sigma_1=\sigma_2=\dots=\sigma_{n-1}$ and
$\sigma_2=\sigma_3=\dots=\sigma_n$, respectively, so these bounds are the best we can get from Proposition \ref{prop:exprsingval}.

By writing $C^2=U/L$, the above gives lower and upper bounds of any of $L$, $U$, and $C$ in terms of any other.

Bounds involving $D$ are more complicated.
We only derive upper bounds of $D$ in terms of $L$, $U$, and $C$ in Proposition \ref{prop:upperbdD}, and not the corresponding lower bounds, as only the upper bounds will be used in later sections.
Then we invert these results to get a lower bound of $L$, an upper bound of $U$, and an upper bound of $C$ in terms of $D$ in Proposition \ref{prop:invertupperbdD}.
The bounds in Proposition \ref{prop:upperbdD} are tight in the sense that an equality condition exists in terms of the singular values, while we lose some tightness in inverting them in Proposition \ref{prop:invertupperbdD}.
Bounds in Proposition \ref{prop:invertupperbdD} will be very useful in Sections \ref{subsec:boundsUpperEB} and \ref{subsec:boundsUpperE}.

\begin{prop}
\label{prop:upperbdD}
Let $(E,B)$ be a basis pair with $\abs{E}=n>1$.
Let $0 < \sigma_1 \leq \sigma_2 \leq \dots \leq \sigma_n$ be the singular values of $T(E,B)$.
Let $L=L_E(B)$, $U=U_E(B)$, $\rho=\rho_E(B)$, and $D=D_E(B)$. Then
$$D \leq \sqrt{L}\paren{\frac{n^2-L}{n-1}}^{\frac{n-1}{2}},$$
where equality holds if and only if $\sigma_2=\sigma_3=\dots=\sigma_n$;
$$D \leq \sqrt{U}\paren{\frac{n^2-U}{n-1}}^{\frac{n-1}{2}},$$
where  equality holds if and only if $\sigma_1=\sigma_2=\dots=\sigma_{n-1}$;
and
$$D \leq \frac{2\sqrt\rho}{\rho+1} n^{n/2},$$
where equality holds if and only if $\sigma_2=\sigma_3=\dots=\sigma_{n-1}=\sqrt{n}$.
\end{prop}

\begin{proof}
By the quadratic mean-geometric mean inequality,
$$D = \sigma_1\sigma_2\dots\sigma_n
\leq \sigma_1 \paren{\frac{\sigma_2^2+\sigma_3^2+\dots+\sigma_n^2}{n-1}}^{\frac{n-1}{2}} = \sqrt{L}\paren{\frac{n^2-L}{n-1}}^{\frac{n-1}{2}},$$
with the equality case as described. Similarly,
$$D = \sigma_1\sigma_2\dots\sigma_n
\leq \paren{\frac{\sigma_1^2+\sigma_2^2+\dots+\sigma_{n-1}^2}{n-1}}^{\frac{n-1}{2}}\sigma_n = \sqrt{U}\paren{\frac{n^2-U}{n-1}}^{\frac{n-1}{2}},$$
with the desired equality case.

The last inequality is more complicated. We claim that if $\rho=\sigma_n^2/\sigma_1^2$ is fixed and the singular values satisfy $\sigma_1^2+\sigma_2^2+\dots+\sigma_n^2 = n^2$, then the maximum value of $D$ is given by the desired expression.
The case of $n=2$ is easily worked out, so assume that $n \geq 3$.
First notice that if we fix $\sigma_1$ and $\sigma_n$, then the maximum of $D$ is achieved when $\sigma_2=\sigma_3=\dots=\sigma_{n-1}$ by the quadratic mean-geometric mean inequality. So suppose that this is the case.
Let $\sigma_1=x$, $\sigma_n=\sqrt{\rho}x$, and the rest of the singular values be
$$\sigma_i = \sqrt{\frac{n^2-(\rho+1)x^2}{n-2}}.$$
Then
$$D^2=C^2x^4 \paren{\frac{n^2-(\rho+1)x^2}{n-2}}^{n-2}.$$
By the arithmetic mean-geometric mean inequality,
$$D^2=\frac{4\rho}{(\rho+1)^2} \bracket{\frac{(\rho+1)x^2}{2}\cdot \frac{(\rho+1)x^2}{2} \cdot \paren{\frac{n^2-(\rho+1)x^2}{n-2}}^{n-2}}
\leq \frac{4\rho}{(\rho+1)^2} n^n,$$
which implies the desired result.
One can check that the equality case is
$\sigma_2=\sigma_3=\dots=\sigma_{n-1}=\sqrt{n}$.
\end{proof}

We now seek to invert the inequalities in Proposition \ref{prop:upperbdD} to bound $L$, $U$, and $\rho$ in terms of $D$.
The bounds obtained are not sharp, but they are sharp within constant factors.

\begin{prop}
\label{prop:invertupperbdD}
Let $(E,B)$ be a basis pair with $\abs{E}=n>1$.
Let $0 < \sigma_1 \leq \sigma_2 \leq \dots \leq \sigma_n$ be the singular values of $T(E,B)$.
Let $L=L_E(B)$, $U=U_E(B)$, $\rho=\rho_E(B)$, and $D=D_E(B)$. Then
$$L > \paren{\frac{n-1}{n^2}}^{n-1}D^2,$$
where for $\sigma_2=\sigma_3=\dots=\sigma_n$, the two sides are within a factor of $\paren{1+\frac{1}{n-1}}^{n-1}< e$ of each other;
$$n^2-U > (n-1)\paren{\frac{D}{n}}^{\frac{2}{n-1}},$$
where for $\sigma_1=\sigma_2=\dots=\sigma_{n-1}$, the two sides are within a factor of $n^\frac{1}{n-1} \leq 2$ of each other; and
$$\rho < \frac{4n^n}{D^2},$$
where for $\sigma_2=\sigma_3=\dots=\sigma_{n-1}=\sqrt{n}$, the two sides are within a factor of $1+1/\rho \leq 2$ of each other.
\end{prop}

\begin{proof}
Apply Proposition \ref{prop:upperbdD}.
Notice that $n^2-L < n^2$ and the two sides are within a factor of $n/(n-1)$ of each other;
$\sqrt{U} < n$ and the two sides are within a factor of $\sqrt{n}$ of each other; and $\rho+1 > \rho$.
\end{proof}

\subsection{Upper bounds for tightness quantities of $(E,B)$}
\label{subsec:boundsUpperEB}

We know that for a basis pair $(E,B)$, $L_E(B)>0$, $U_E(B)<n^2$, $\rho_E(B)<\infty$, and $D_E(B)>0$.
In this subsection, we give quantitative bounds on how close these quantities can get to these extremes without collapsing to 0, $n^2$, or $\infty$. This can be interpreted as how far $(E,B)$ can be from being spectral without ceasing to be a basis pair.

The estimates obtained have to depend on both $\abs{E}=n$ and the ambient space $G$, as Example \ref{ex:ebfarspec} demonstrates. However, Theorem \ref{thm:upperbdcondeb} shows that they depend on $G$ in an interesting way.
Specifically, they depend on \emph{which} prime powers are present in the $n_i$ in the decomposition $G \cong \Z_{n_1}\times \Z_{n_2} \times\dots\times \Z_{n_\ell}$, but not \emph{the number of times} the prime powers appear.
For example, for $G=\Z_m^d$, the estimates do not depend on $d$ at all.

This phenomenon may be partly explained by the affine restriction property (Prop. \ref{prop:affinerestrict}).
Indeed, for $E \subseteq \Z_m^d$, as $d$ grows but $\abs{E}=n$ and $m$ stay fixed, it is likely that $E$ will lie in a coset of a ``small'' direct summand of $G$ whose size depends on $n$ and $m$ but not $d$.
If this is true, then by the affine restriction property, $d$ plays no role in the growth of $\rho_E(B)$.

Miraculously, in Section \ref{subsec:boundsUpperE}, we will show that for quantities $L(E)$, $U(E)$, $\rho(E)$, and $D(E)$, the dependence on the group $G$ can be completely eliminated, so that estimates on these quantities only depend on $\abs{E}=n$.
The intuition behind this may be the more-space intuition mentioned in Section \ref{subsec:boundsNearlySpec}.
As $G$ grows, the individual pair $(E,B)$ may get further from being spectral, but there are more choices of $B$ to choose from, so in the end the set $E$ itself is not so far from being spectral.

We now outline the strategy of proofs in this section. Bounding $\rho_E(B)$ away from infinity (or $L_E(B)$ away from zero, or $U_E(B)$ away from $n^2$) is a difficult task.
However, as $D_E(B)$ is the absolute value of an integer combination of roots of unity, it is easier to bound this quantity away from zero.
Specifically, we use a standard argument by permuting roots of unity in Lemma \ref{lem:lowerbdrootunity}.
This directly gives an estimate on $D_E(B)$.
Then we translate this into estimates on other tightness quantities via Proposition \ref{prop:invertupperbdD}.

Let $\phi$ be Euler's totient function, that is, $\phi(n)$ counts the number of integers $1 \leq k \leq n$ that are relatively prime to $n$.

\begin{lemma}
\label{lem:lowerbdrootunity}
Let $m$ be a positive integer, and $\omega$ be a primitive $m$th root of unity. Let $P(x) \in \Z[x]$ be such that
$P(\omega) \neq 0$.
Suppose that $\abs{P(\omega^k)} \leq C$ for all $1 \leq k \leq m$ that are relatively prime to $m$. Then $\abs{P(\omega)} \geq C^{1-\phi(m)}$.
\end{lemma}

\begin{proof}
Because $P(\omega) \neq 0$, for any $k$ that is relatively prime to $m$, $P(\omega^k) \neq 0$. Thus
$$S = \prod_{1 \leq k \leq m, \gcd(k,m)=1} P(\omega^k) \neq 0.$$
Notice that $S$ is the value of a symmetric polynomial evaluated at $\omega^k$, $1 \leq k \leq m$ with $\gcd(k,m)=1$.
Since these $\omega^k$ are roots of a monic polynomial with integer coefficients, $S$ must be an integer. Hence $\abs{S} \geq 1$. It follows that
$$\abs{P(\omega)} = \frac{\abs{S}}{\prod_{2 \leq k \leq m, \gcd(k,m)=1} \abs{P(\omega^k)}} \geq \frac{1}{C^{\phi(m)-1}}.$$
\end{proof}

\begin{remark}
Lemma \ref{lem:lowerbdrootunity} is a generalization of a standard argument used to find a lower bound for a nonzero sum of $2m$-th roots of unity. To see the connection, notice that a nonzero sum of $N$ $2m$-th roots of unity can be written as $P(\omega)$ for some $P(x) \in \Z[x]$, where the sum of coeffients of $P(x)$ is at most $N$.
Then $\abs{P(\omega^k)} \leq N$.
So the lemma implies that this sum has absolute value at least $N^{1-\phi(m)}$.
See \cite{My} for further results on this problem.
\end{remark}

Let $G$ be a finite group, written in multiplicative notation. Following \cite[p. 202]{Ro}, the \emph{minimal exponent} of $G$ is the smallest positive integer $m$ such that $g^{m}=1$ for all $g\in G$.
In other words, it is the least common multiple of the orders of all the elements of $G$. If $G=\Z_{n_1}\times\Z_{n_2}\times\dots\times\Z_{n_\ell}$, then the minimal exponent of $G$ is the least common multiple of $n_{1}, \dots, n_{\ell}$. Hence, the minimal exponent of a group is the number $M$ in the proposition below, but this number can be defined without reference to the specific way that $G$ is decomposed into a direct product of cyclic groups.

\begin{theorem}
\label{thm:upperbdcondeb}
Let $G=\Z_{n_1}\times\Z_{n_2}\times\dots\times\Z_{n_\ell}$ and $M=\lcm(n_1,n_2,\dots,n_\ell)$.
Let $E \subseteq G$ with $\abs{E}=n$. If $(E,B)$ is a basis pair, then
$D_E(B) \geq n^{n(1-\phi(M))/2}$. Consequently,
$$\rho_E(B) < 4n^{n\phi(M)}.$$
\end{theorem}

\begin{proof}
The results for $n=1$ are obvious, so assume that $n>1$.
In a similar manner to the proof of Proposition \ref{prop:scalespec}, let $M(x)$ be the $n \times n$ matrix whose $(i,j)$ entry is
$$x^{\paren{\frac{(x_i)_1 (b_j)_1}{n_1}+\frac{(x_i)_2 (b_j)_2}{n_2}+\dots+\frac{(x_i)_\ell (b_j)_\ell}{n_\ell}}M}.$$
(This is slightly different from \eqref{eq:entryMx}.)
Let $P(x) = \det M(x) \in \Z[x]$.
If $\omega = \chi(1/M)$, then for every $1 \leq k \leq M$ that is relatively prime to $M$,
$T(E,kB) = M(\omega^k)$. Hence
$$D_E(kB) = \abs{\det M(\omega^k)}=\abs{P(\omega^k)}.$$
By Proposition \ref{prop:simplebds}, $\abs{P(\omega^k)} \leq n^{n/2}$, so Lemma \ref{lem:lowerbdrootunity} yields the desired bound for $D_E(B)$.
Finally, we apply Proposition \ref{prop:invertupperbdD} to obtain the bound for $\rho_E(B)$.
\end{proof}

From Theorem \ref{thm:upperbdcondeb}, we can also derive a lower bound for $L_E(B)$ and an upper bound for $U_E(B)$ via Proposition \ref{prop:invertupperbdD}.

The crucial point is that the bounds obtained in Theorem \ref{thm:upperbdcondeb} depend on $G$ even for fixed $\abs{E}=n$.
The example below shows that this must be the case.

\begin{example}
\label{ex:ebfarspec}
Let $E=B=\set{0,1} \subseteq \Z_p$, $p$ a prime. Then
$$T(E,B) = \begin{pmatrix} 1 & 1 \\ 1 & \chi(1/p) \end{pmatrix}.$$
It is easily checked that $L_E(B) \to 0$, $U_E(B) \to 4$, $\rho_E(B) \to \infty$, and $D_E(B) \to 0$ as $p \to \infty$.
Therefore, $(E,B)$ gets arbitrarily far away from being spectral as $p \to \infty$.
\end{example}

\subsection{Upper bounds for tightness quantities of $E$}
\label{subsec:boundsUpperE}

Finally, in this subsection, we derive a bound on how far from being spectral a set $E$ can be.
The arguments are similar to those used in Section \ref{subsec:boundsUpperEB}, but the new idea is the following.
We will demonstrate that, if $B$ is carefully picked, then we can make $(E,B)$ close to being spectral independent of the ambient group $G$, even though an individual pair $(E,B)$ may be far from being spectral.
Specifically, we consider $kB$ for all $k$ such that multiplication by $k$ is an invertible $\Z$-linear transformation on $G$.
These pairs average each other out, and one of $(E,kB)$ must be rather close to being spectral, as shown in Lemma \ref{lem:lowerbdrootunitykB}.
Then estimates on tightness quantities follow as before.

The results of this section show that a weaker sense of the more-space intuition in Section \ref{subsec:boundsNearlySpec} holds true.
That is, for $\abs{E}=n$ fixed and as $G$ grows, there are more choices of $B$ to choose from, so one of $(E,B)$ must be quite close to being spectral.
In this case, it turns out to suffice to pick a random $B$ and consider $kB$ over all $k$.

The takeaway of these results is the following.
To get a sequence of sets $E$ that are further and further away from being spectral, the size of $E$ has to increase to infinity, no matter which sequence of $G$ we choose.

We begin with a variant of Lemma \ref{lem:lowerbdrootunity}.

\begin{lemma}
\label{lem:lowerbdrootunitykB}
Let $m$ be a positive integer, and $\omega$ be a primitive $m$th root of unity. Let $P(x) \in \Z[x]$ be such that $P(\omega) \neq 0$.
Then there is a $1 \leq k \leq m$ that is relatively prime to $m$ such that $\abs{P(\omega^k)} \geq 1$.
\end{lemma}

\begin{proof}
Consider the quantity $S$ defined in Lemma \ref{lem:lowerbdrootunity}. We have $\abs{S} \geq 1$. This implies that there is a $1 \leq k \leq m$, $\gcd(k,m)=1$, such that $\abs{P(\omega^k)} \geq 1$.
\end{proof}

The next proposition shows that for any basis pair $(E,B)$, one of $(E,kB)$ is rather close to being spectral.

\begin{prop}
\label{prop:upperBdPairLoopAround}
Let $G=\Z_{n_1}\times\Z_{n_2}\times\dots\times\Z_{n_\ell}$ and $M=\lcm(n_1,n_2,\dots,n_\ell)$.
Let $E \subseteq G$ with $\abs{E}=n>1$. Suppose that $(E,B)$ is a basis pair. Then there is a $1 \leq k \leq M$ that is relatively prime to $M$ such that
$D_E(kB) \geq 1$. Consequently, for this $k$,
\begin{alignat*}{2}
L_E(kB) &> \paren{\frac{n-1}{n^2}}^{n-1} &&> \frac{1}{en^{n-1}}, \\
U_E(kB) &< n^2 - \frac{n-1}{n^\frac{2}{n-1}} &&\leq n^2-\frac{n-1}{4}, \\
\rho_E(kB) &< 4n^n.
\end{alignat*}
\end{prop}

\begin{proof}
An analogous argument to the proof of Theorem \ref{thm:upperbdcondeb}, but with Lemma \ref{lem:lowerbdrootunitykB} in place of Lemma \ref{lem:lowerbdrootunity}, gives the estimate $D_E(kB) \geq 1$.
Now apply Proposition \ref{prop:invertupperbdD} to obtain the rest of the estimates.
\end{proof}

Notice that in Proposition \ref{prop:upperBdPairLoopAround}, the second steps in the bounds for $L_E(kB)$ and $U_E(kB)$ lose only at most constant factors.

We now present our main result on estimates on tightness quantities of $E$, which are independent of the group $G$.

\begin{theorem}
\label{thm:upperBdSetBySize}
Let $E \subseteq G$ with $\abs{E}=n>1$.
Then
$$L(E) > \frac{1}{en^{n-1}}, \quad U(E) < n^2-\frac{n-1}{4}, \quad \rho(E) < 4n^n, \quad \text{and} \quad D(E) \geq 1.$$
In particular, these estimates are independent of the group $G$.
\end{theorem}

\begin{proof}
Follows directly from the existence of a basis pair (Cor. \ref{cor:basispairexist}) and Proposition \ref{prop:upperBdPairLoopAround}.
\end{proof}

\section{Decomposition}
\label{sec:decompose}

Let $G=H \oplus K$.
In this section, we investigate what happens when a set $E \subseteq H$ has its ``dimension'' increased to become the set $E \times K \subseteq G$.
An interesting conclusion of Section \ref{subsec:decomposeCartesian} is that $E \times K$ is not further from being spectral than $E$.
Specifically, $\rho(E \times K) \leq \rho(E)$, and similarly for other tightness quantities.
Even more strikingly, it will later turn out in Section \ref{subsec:decomposeExact} that
$\rho(E \times K) = \rho(E)$, and similarly for other tightness quantities. Thus, $E$ and $E \times K$ are exactly as far from being spectral as one another.

We first need some preliminary definitions.

\begin{definition}
\label{def:normalizedquants}
For an equal-size pair $(E,B)$, define the \emph{normalized tightness quantities} as follows:
\begin{itemize}
\item $\widetilde{L}_E(B) = n/L_E(B)$,
\item $\widetilde{U}_E(B) = U_E(B)/n$,
\item $\widetilde{\rho}_E(B) = \rho_E(B)$,
\item $\widetilde{D}_E(B) = \sqrt{n}/D_E(B)^{1/n}.$
\end{itemize}
Notice that the normalized version of $\rho$ is itself.
\end{definition}

Let $Q$ denote any tightness quantity
and let $\widetilde{Q}$ be its normalized version.
Then $\widetilde{Q}_E(B) \geq 1$ and it is one if and only if $(E,B)$ is a spectral pair.
Accordingly, define
$$\widetilde{Q}(E) = \min_B \widetilde{Q}_E(B).$$
Notice that $\widetilde{Q}(E)=1$ if and only if $E$ is a spectral set.
The relationship between $Q(E)$ and $\widetilde{Q}(E)$
is the same as the relationship between $Q_E(B)$
and $\widetilde{Q}_E(B)$ in Definition \ref{def:normalizedquants}.

\subsection{Cartesian products}
\label{subsec:decomposeCartesian}

In this subsection, we study basic properties of $(E,B)$ when $E$ and $B$ are Cartesian products. Later, in Section \ref{subsec:decomposeMultiTile}, we will study $E$ which are multi-tiles by subgroups, generalizing Cartesian products where one factor is the full direct summand.

Specifically, let $H$ and $K$ be finite abelian groups, and let $G=H \oplus K$.
For $E_1 \subseteq H$ and $E_2 \subseteq K$, we consider the Cartesian product $E=E_1 \times E_2 \subseteq G$.
Notice that this is also the Minkowski sum $E=E_1+E_2$ under the usual identification.

Recall that $\widehat{G} \cong \widehat{H} \times \widehat{K}$, where for any $\hat{h} \in \widehat{H}$ and $\hat{k} \in \widehat{K}$, $(\hat{h},\hat{k})$ acts as an element of $\widehat{G}$ by $(\hat{h},\hat{k})(h,k) = \hat{h}(h)\hat{k}(k)$ for any $h \in H$ and $k \in K$.
So for any $B_1 \subseteq \widehat{H}$ and $B_2 \subseteq \widehat{K}$, we can identify $B = B_1 \times B_2$ with a subset of $\widehat{G}$, and we will do so without further comments.
Notice that since both $B$ and the set $E$ from above are Cartesian products, there is duality between $E$ and $B$ in our situation.

If $(E_1,B_1)$ and $(E_2,B_2)$ are equal-size pairs,
then $(E,B)$ is also an equal-size pair.
Our goal in this subsection is to show that the three pairs are closely related.
We start with the next proposition, which relates
their Fourier matrices.

\begin{prop}
\label{prop:TEBtensor}
Let $(E_1, B_1) \subseteq H \times \widehat{H}$ and
$(E_2,B_2) \subseteq K \times \widehat{K}$ be equal-size pairs.
Let $E=E_1 \times E_2$ and $B=B_1 \times B_2$.
Then, up to an ordering of rows and columns,
$$T(E,B) = T(E_1,B_1) \otimes T(E_2,B_2),$$
where $\otimes$ is the Kronecker product.
\end{prop}

\begin{proof}
The entry of $T(E,B)$ in the column corresponding to $(b_1,b_2)$, $b_i \in B_i$, and the row corresponding to $(x_1,x_2)$, $x_i \in E_i$, is given by 
$$(b_1,b_2)(x_1,x_2) = b_1(x_1) b_2(x_2).$$
On the other hand, the entry of $T(E_i,B_i)$ in the column corresponding to $b \in B_i$ and the row corresponding to $x \in E_i$ is $b(x)$.
So the result follows from the definition of the Kronecker product.
\end{proof}

Let $A$ and $B$ be square matrices,
and let $\set{\sigma_i}_{i=1}^n$ and $\set{\tau_j}_{j=1}^m$ be the singular values of $A$ and $B$, respectively, counting multiplicities.
It is known (see \cite[Thm. 4.2.15]{HJ}) that the singular values of the Kronecker product $A \otimes B$ are exactly
$\set{\sigma_i \tau_j}_{1 \leq i\leq n, 1 \leq j \leq m}$, counting multiplicities.
Using this, the tightness quantities of the three pairs can be related to one another, as the proposition below shows.

\begin{prop}
\label{prop:Qcartesianpair}
Let $(E_1, B_1) \subseteq H \times \widehat{H}$ and
$(E_2,B_2) \subseteq K \times \widehat{K}$ be equal-size pairs.
Let $E=E_1 \times E_2$ and $B=B_1 \times B_2$.
Then, for any normalized tightness quantity $\widetilde{Q}$,
$$\widetilde{Q}_E(B) = \widetilde{Q}_{E_1}(B_1)\widetilde{Q}_{E_2}(B_2).$$
\end{prop}

\begin{proof}
By Proposition \ref{prop:TEBtensor}, the singular values of $T(E,B)$ are products of singular values of $T(E_1,B_1)$ and $T(E_2,B_2)$, counting multiplicities.
The result now follows by writing $\widetilde{Q}$ in terms of singular values using Proposition \ref{prop:exprsingval}.
\end{proof}

Proposition \ref{prop:Qcartesianpair} allows us to bound tightness quantities of $E$ in terms of the corresponding quantities of $E_1$ and $E_2$, as shown in Corollary \ref{cor:bdQcartesianset} below.
In that corollary, $E_1 \subseteq H \subseteq G$, but by the affine restriction property (Prop. \ref{prop:affinerestrict}), $\widetilde{Q}(E_1;H) = \widetilde{Q}(E_1;G)$, so we need not specify the ambient group. A similar remark applies to $E_2$.

\begin{cor}
\label{cor:bdQcartesianset}
Let $E_1 \subseteq H$, $E_2 \subseteq K$,
and $E=E_1 \times E_2$.
Then, for any normalized tightness quantity $\widetilde{Q}$,
$$\widetilde{Q}(E) \leq \widetilde{Q}(E_1)\widetilde{Q}(E_2).$$
In particular, if $E_1$ and $E_2$ are spectral, then
$E$ is spectral.
\end{cor}

\begin{proof}
Pick $B_1 \subseteq \widehat{H}$ and $B_2 \subseteq \widehat{K}$ such that $\widetilde{Q}_{E_i}(B_i) = \widetilde{Q}(E_i)$.
By Proposition \ref{prop:Qcartesianpair},
$\widetilde{Q}_E(B_1 \times B_2) = \widetilde{Q}(E_1)\widetilde{Q}(E_2)$ with $B_1 \times B_2 \subseteq \widehat{G}$, so that
$\widetilde{Q}(E) \leq \widetilde{Q}(E_1)\widetilde{Q}(E_2)$.
\end{proof}

Corollary \ref{cor:bdQcartesianset} provides a nice generalization of the well-known fact that a Cartesian product of spectral sets is spectral; an elementary proof of this is obtained by multiplying the relevant characters to form an orthogonal basis.
An important special case of the corollary occurs below when $E_2$ is spectral, for example, when $E_2$ is the whole group $K$.

\begin{cor}
\label{cor:bdQsetprodspec}
Let $E_1 \subseteq H$ and $E_2 \subseteq K$, where $E_2$ is spectral. Let $E=E_1 \times E_2$.
Then $E$ is not further from being spectral than $E_1$. Specifically, for any normalized tightness quantity $\widetilde{Q}$,
$$\widetilde{Q}(E) \leq \widetilde{Q}(E_1).$$
In particular, $\rho(E_{1}\times E_{2})\leq \rho(E_{1})$ when $E_{2}$ is spectral.
\end{cor}

As an application, we construct in Example \ref{ex:cartclosespec} below large sets $E$ that do not tile but are very close to being spectral.
The significance of such an example is the following.
It is known that a set $E \subseteq G$ that tiles by a subgroup of $G$ is spectral since this implies that $E$ tiles by a lattice. The proof that tiling by a lattice implies that $E$ is spectral is given in \cite{BHM}.
We will also prove that a set $E$ is spectral if it tiles by a subgroup in Proposition \ref{prop:TilesAreSpec}.
Recall, from Theorem \ref{thm:upperBdSetBySize}, that small sets cannot be very far from being spectral.
Thus, the example below demonstrates that good spectral behavior can occur apart from the cases of being small and tiling by a direct summand.
Later in Section \ref{sec:MultiTileGeometricComplex}, we will discuss more about the relationship between (multi-)tiling and being close to spectral.

\begin{example}
\label{ex:cartclosespec}
Let $p$ be a prime and let $G=\Z_p^2$. Let $E_p=\set{0,1} \times \Z_p$. By Corollary \ref{cor:bdQsetprodspec} and Proposition \ref{prop:size2closespec},
$$\rho(E_p) \leq \rho(\set{0,1};\Z_p) \to 1$$
as $p \to \infty$.
So $E_p$ is close to being spectral for large $p$.
\end{example}

An interesting question is whether the inequality in Corollary \ref{cor:bdQcartesianset} can be strict.
In the special case where $E$ is spectral, this asks whether it is necessary that $E_1$ and $E_2$ are also spectral. This problem about spectral sets is open, although it has been proved true in special cases. These include the special cases in which one of the factors is an interval or convex polygon, due to results of Greenfeld and Lev \cite{GnL1,GnL3}, and the special case where one of the factors is a union of two intervals, due to a result of Kolountzakis \cite{Ko2}. For the above question about the inequality's strictness, however, the answer is false for all $\widetilde{Q}$ except perhaps $\widetilde{D}$, as illustrated by the following example.
We do not know the answer to this question for $\widetilde{Q}=\widetilde{D}$.

\begin{example}
Let $H=K=\Z_3$, $E_1=E_2 = \set{0,1} \subseteq \Z_3$, and $E=E_1 \times E_2$. It is easily checked that
$$\widetilde{L}(E_1)= 2,\quad
\widetilde{U}(E_1)= \frac{3}{2},\quad
\widetilde{\rho}(E_1)= 3,\quad
\widetilde{D}(E_1)= \frac{\sqrt{2}}{\sqrt[4]{3}},$$
while it can be checked by a computer (using e.g. our MATLAB program linked at the end of Section \ref{subsec:contributions}) that
$$\widetilde{L}(E) \doteq 3.490711985,\quad
\widetilde{U}(E) \doteq \frac{3}{2},\quad
\widetilde{\rho}(E) \doteq \paren{\frac{1+\sqrt{5}}{2}}^4,\quad
\widetilde{D}(E) \doteq \frac{2}{\sqrt{3}}.$$
Here, $\doteq$ means the equalities only hold numerically, i.e., up to small roundoff errors.
Thus, the inequality in Corollary \ref{cor:bdQcartesianset} can be strict for $E$ and all $\widetilde{Q}$ except perhaps $\widetilde{D}$.
\end{example}

Finally, we remark that Propositions \ref{prop:TEBtensor} and \ref{prop:Qcartesianpair} and Corollary \ref{cor:bdQcartesianset} can be easily generalized to the case of $n$-fold Cartesian products for any $n$.

\subsection{Multi-tiles}
\label{subsec:decomposeMultiTile}

We now investigate sets $E$ that are more general than Cartesian products where one factor is the full direct summand: multi-tiles by subgroups.

Our setup is the following.
Let $G$ be a finite abelian group, and $H \subseteq G$ be a subgroup.
Let $K:=G/H=\set{k_1,\dots,k_m}$.
Let $E \subseteq G$ multi-tile $G$ with partner $H$ at level $\ell$.
Let $F_i := E \cap k_i$.
Then this multi-tiling property is equivalent to $\abs{F_i}=\ell$ for every $i$.
Call $F_i$ the \emph{cross sections} of $E$ with respect to $H$.
For any choice of $g_i \in k_i$, we say that $F_i' := F_i-g_i \subseteq H$ are the \emph{translated $H$-cross sections} of $E$, which are unique up to $H$-translation. See Figure \ref{fig:multiTilSetup}.

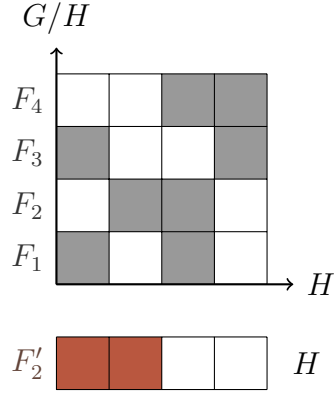
\begin{figure}
\begin{center}
\begin{tikzpicture}[scale=0.7]
\fill[gray!80!white] (0,0) rectangle (1,1);
\fill[gray!80!white] (2,0) rectangle (3,1);
\fill[gray!80!white] (1,1) rectangle (2,2);
\fill[gray!80!white] (2,1) rectangle (3,2);
\fill[gray!80!white] (0,2) rectangle (1,3);
\fill[gray!80!white] (3,2) rectangle (4,3);
\fill[gray!80!white] (2,3) rectangle (3,4);
\fill[gray!80!white] (3,3) rectangle (4,4);
\foreach \x in {1,2,3,4}
\node at (-0.55,\x-0.5) {\textcolor{gray!40!black}{$F_{\x}$}};
\draw (0,0) grid (4,4);
\draw[thick,->] (0,0) -- (4.5,0); \node[anchor=west] at (4.55,0) {$H$};
\draw[thick,->] (0,0) -- (0,4.5); \node[anchor=south] at (0,4.55) {$G/H$};
\fill[Mahogany!80!white] (0,-2) rectangle (1,-1);
\fill[Mahogany!80!white] (1,-2) rectangle (2,-1);
\node at (-0.55,-1.5) {\textcolor{Mahogany!40!black}{$F_2'$}};
\node at (4.75,-1.5) {$H$};
\draw (0,-2) grid (4,-1);
\end{tikzpicture}
\end{center}
\caption{A multi-tile as in our setup ($\ell=2$, $m=4$).}
\label{fig:multiTilSetup}
\end{figure}

In Theorem \ref{thm:decomposeMultiTile}, we show that such sets $E$ behave well spectrally compared to the least spectral $F_i'$ against a common basis partner. Compare with results of Section \ref{subsec:decomposeCartesian}.
Therefore, apart from Cartesian products, this provides us with another way to construct large sets that do not tile and are not far from being spectral.

The basis partner $B$ will be a ``Cartesian product'' constructed as follows. Let $B_H \subseteq \widehat{H}$ have $\ell$ elements. Under the canonical isomorphism $\widehat{H}\cong \widehat{G}/H^{\perp}$, we associate, to each $\phi\in B_H$, a (non-unique) element $\widetilde{\phi}\in \hat{g}H^{\perp}$, where $\hat{g}H^{\perp}$ is the coset corresponding to $\phi$ under the canonical isomorphism. Let
$$\widetilde{B_H}:=\{\widetilde{\phi}\}_{\phi\in B_H} \subseteq \widehat{G}.$$
Define
$$B:=\{\phi\psi: \phi \in \widetilde{B_H}, \psi \in H^\perp \}.$$
Then $(E,B)$ is an equal-size pair. See Figure \ref{fig:basisPartnerSetup}.

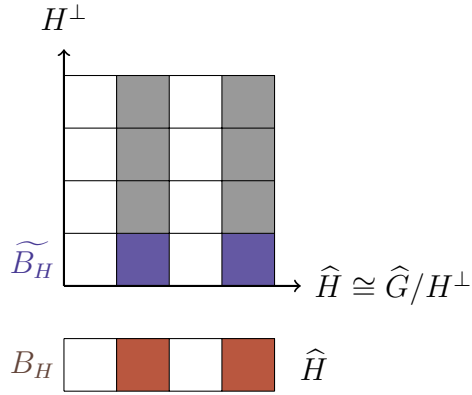
\begin{figure}
\begin{center}
\begin{tikzpicture}[scale=0.7]
\fill[BlueViolet!80!white] (1,0) rectangle (2,1);
\fill[BlueViolet!80!white] (3,0) rectangle (4,1);
\fill[gray!80!white] (1,1) rectangle (2,2);
\fill[gray!80!white] (3,1) rectangle (4,2);
\fill[gray!80!white] (1,2) rectangle (2,3);
\fill[gray!80!white] (3,2) rectangle (4,3);
\fill[gray!80!white] (1,3) rectangle (2,4);
\fill[gray!80!white] (3,3) rectangle (4,4);
\node at (-0.6,0.5) {\textcolor{BlueViolet}{$\widetilde{B_H}$}};
\draw (0,0) grid (4,4);
\draw[thick,->] (0,0) -- (4.5,0); \node[anchor=west] at (4.55,0) {$\widehat{H} \cong \widehat{G}/H^\perp$};
\draw[thick,->] (0,0) -- (0,4.5);
\node[anchor=south] at (0,4.65) {$H^\perp$};
\fill[Mahogany!80!white] (1,-2) rectangle (2,-1);
\fill[Mahogany!80!white] (3,-2) rectangle (4,-1);
\node at (-0.6,-1.5) {\textcolor{Mahogany!40!black}{$B_H$}};
\node at (4.75,-1.5) {$\widehat{H}$};
\draw (0,-2) grid (4,-1);
\end{tikzpicture}
\end{center}
\caption{A basis partner as in our setup for the set in Figure \ref{fig:multiTilSetup}.}
\label{fig:basisPartnerSetup}
\end{figure}

\begin{theorem}
\label{thm:decomposeMultiTile}
Let $E$ and $B$ be as above.
Then, for $\widetilde{Q} = \widetilde{L}$ or $\widetilde{Q} = \widetilde{U}$,
$$\widetilde{Q}_E(B) = \max_{1 \leq i \leq m} \widetilde{Q}_{F_i'}(B_H).$$
For $\rho$,
$$\rho_E(B) = \max_{1 \leq i \leq m} \widetilde{L}_{F_i'}(B_H)
\max_{1 \leq i \leq m} \widetilde{U}_{F_i'}(B_H)
\geq \max_{1 \leq i \leq m} \rho_{F_i'}(B_H).$$
For $\widetilde{D}$,
$$\widetilde{D}_E(B) = \paren{\prod_{i=1}^m \widetilde{D}_{F_i'}(B_H)}^{1/m}.$$
(Note by translation invariance that $\widetilde{Q}_{F_i'}(B_H)$ is well-defined since $F_i'$ is unique up to $H$-translation.)
\end{theorem}

\begin{proof}
Write $H^\perp = \{\hat{k}_1,\hat{k}_2,\dots,\hat{k}_m\}$.
With an appropriate reordering of rows and columns, $T(E,B)$ is an $m \times m$ block matrix of $\ell \times \ell$ blocks, where the $(i,j)$ block corresponds to $\{\widetilde{\phi}\cdot\hat{k}_j\}_{\phi\in B_{H}}$ acting on $F_i$.
So the $(i,j)$ block of $T(E,B)$ is $T(F_i,\widetilde{B_H})\hat{k}_j(g_i)$, since $\hat{k}_j \in H^\perp$ and $F_i \subseteq H + g_i$.
Thus, we can factor
\begin{multline*}
T(E,B) = \begin{pmatrix}
T(F_1,\widetilde{B_H}) \\
& T(F_2,\widetilde{B_H}) \\
& & \ddots \\
& & & T(F_m,\widetilde{B_H})\end{pmatrix} \cdot \\
\begin{pmatrix}
\hat{k}_1(g_1)I_\ell & \hat{k}_2(g_1)I_\ell & \dots & \hat{k}_m(g_1)I_\ell \\
\hat{k}_1(g_2)I_\ell & \hat{k}_2(g_2)I_\ell & \dots & \hat{k}_m(g_2)I_\ell \\
\vdots & \vdots & \ddots & \vdots \\
\hat{k}_1(g_m)I_\ell & \hat{k}_2(g_m)I_\ell & \dots & \hat{k}_m(g_m)I_\ell
\end{pmatrix}.
\end{multline*}
Call the first factor $T_1$ and the second factor $T_2$.

For the analysis of $T_1$, observe that $T_{1}$'s singular values are the union of the singular values of the $T(F_i,\widetilde{B_H})$, counting multiplicities. By translation invariance, $T(F_i,\widetilde{B_H})$ has the same singular values as $T(F_i',\widetilde{B_H})$, which is equal to $T(F_i',B_H)$ by construction. Hence
\[
\norm{T_1} =  \max_{1\leq i \leq m} \norm{T(F_i',B_H)},\quad \norm{T_1^{-1}} =  \max_{1\leq i \leq m} \norm{T(F_i',B_H)^{-1}},\quad
\abs{\det T_1} = \prod_{i=1}^m \abs{\det T(F_i',B_H)}.
\]

Now, up to a permutation of rows and columns,
$T_2$ is an $\ell \times \ell$ block diagonal matrix with blocks of size $m\times m$.
Specifically, arrange the rows and columns in the order
\[
1,\ell+1,\dots,(m-1)\ell+1, \hskip 0.5em 2,\ell+2,\dots,(m-1)\ell+2,\hskip 0.5em \dots,\hskip 0.5em \ell,2\ell,\dots,m\ell.
\]
Then all diagonal blocks equal $T(\set{g_i}_{i=1}^m,H^\perp)=[\hat{k}_j(g_i)]_{ij}$.
Under the canonical isomorphism $\widehat{G/H} \cong H^\perp$, each $g_i$ representing an element of $G/H$, this matrix can be thought of as $T(K,\widehat{K})$, i.e. it is unitary.
Hence $T_2$ is unitary.

We conclude that
$\norm{T} = \norm{T_1}$, $\norm{T^{-1}} = \norm{T_1^{-1}}$,
and $\abs{\det T} = \abs{\det T_1}$.
These together with the basic characterizations of tightness quantities (Prop. \ref{prop:exprrieszconst}) imply the desired result.
\end{proof}

\begin{remark}
In his proof that a bounded set $\Omega$ that multi-tiles $\mathbb{R}^{d}$ by a lattice $\Lambda$ has the property that $L^{2}(\Omega)$ has a Riesz exponential basis, Kolountzakis \cite{Ko} gives the factorization
\[
\left[e\big(a_{j}\cdot (x-\lambda_{r}(x))\big)\right]_{1\leq j,r\leq k}
=
\left[e\big(-a_{j}\cdot\lambda_{r}(x)\big)\right]_{1\leq j, r\leq k}
\text{diag}[e(a_{1}\cdot x),\dots, e(a_{k}\cdot x)]
\]
(cf. Equation (14) in his proof of Lemma 2). This is analogous to our factorization of $T(E,B)$ in the proof of Theorem \ref{thm:decomposeMultiTile} into $T_1$ and $T_2$. In fact, quantitatively calculating the constant $C_2$ of Kolountzakis's Lemma 2 yields
\[
C_{2}=kA_{2}=k\cdot\max_{\text{distinct }N(x)}\| N(x)^{-1}\|,
\]
analogous to our above calculation $\| T_{1}^{-1}\| = \max_{1\leq i\leq m}\|T(F_i', B_H)\|$.
A similar remark applies for his constant $C_1$.
\end{remark}

\begin{remark}
\label{rem:generalizeThmDecomposeMultiTile}
In fact, Theorem \ref{thm:decomposeMultiTile} can be generalized a bit further. Consider the set $E$ that is a ``partial multi-tile'' by $H$. Specifically, consider as $E$ some union of $m' \leq m$ cross-sections $F_i$ as above, and consider a set $B_K \subseteq \widehat{K}$ of size $m'$. Define $\widetilde{B_K} \subseteq H^\perp$ of size $m'$ corresponding to $B_K$ via the canonical isomorphism $\widehat{K} \cong H^\perp$. Then define $B := \widetilde{B_H} \cdot \widetilde{B_K}$ as before.

We can then factor $T(E,B) = T_1 T_2$ as in the above proof, with $T_2$ being diagonal with identical blocks $T(\pi(E), B_K)$, where $\pi:G \to G/H$ is the canonical projection. Even though $T_2$ is not always unitary, we can still obtain results using $$\norm{T}\leq\norm{T_1}\norm{T_2},\quad \norm{T^{-1}}\leq\norm{T_1^{-1}}\norm{T_2^{-1}}, \quad \det T = \det T_1 \det T_2.$$
So we can relate $\widetilde{Q}_E(B)$ to $\widetilde{Q}_{F_i'}(B_H)$ and $\widetilde{Q}_{\pi(E)}(B_K)$.

We will not use this result in our main Theorem \ref{thm:mainUpperBdLevelComplexity}, but we think it is nice to note that this generalization exists.
\end{remark}
We now consider the special case where all $F_i$ are translates of one another.
In this case, Theorem \ref{thm:decomposeMultiTile} simplifies as follows.

\begin{cor}
\label{cor:decomposeSumWithTile}
Let $(F,B_H) \subseteq H \times \widehat{H}$ be an equal-size pair.
Let $T \subseteq G$ tile $G$ by $H$ (at level $1$).
Let $E := F+T$, and $B:=B_H \times H^\perp$. Then for any normalized tightness quantity $\widetilde{Q}$,
$$\widetilde{Q}_E(B) = \widetilde{Q}_F(B_H).$$
\end{cor}

This result generalizes Proposition \ref{prop:Qcartesianpair} in the case that $(E_2,B_2) = (K,\widehat{K})$.

\begin{remark}
Using the result in Remark \ref{rem:generalizeThmDecomposeMultiTile}, we can generalize Corollary \ref{cor:decomposeSumWithTile} to the case that $T$ is a ``partial tile'' by $H$, that is, $T$ has at most 1 element in each coset of $H$. However, the results obtained weaken from equalities to inequalities.
\end{remark}

As an application, we will obtain the following known result.

\begin{prop}
\label{prop:TilesAreSpec}
Let $H \subseteq G$ be a subgroup.
Then any set $E\subseteq G$ that tiles with tiling partner $H$ is spectral.
\end{prop}

\begin{proof}
By decomposing $E=\set{0}+E$ and using the fact that $\set{0}$ is trivially spectral, the result follows from Corollary \ref{cor:decomposeSumWithTile}.
\end{proof}

As a second application, we will
construct in Example \ref{ex:multiTileNearlySpec} large sets that do not tile and are also not Cartesian products, but are nearly spectral.
Compare with Example \ref{ex:cartclosespec}.

\begin{example}
\label{ex:multiTileNearlySpec}
Let $p \geq 3$ be a prime and let $G=\Z_p^2$.
Let
$$E_p=\paren{\set{0,1} \times (\Z_p \setminus \set{0})} \cup \set{(1,0),(2,0)}.$$
By Corollary \ref{cor:decomposeSumWithTile} and Proposition \ref{prop:size2closespec},
$$\rho(E_p) \leq \rho(\set{0,1};\Z_p) \to 1$$
as $p \to \infty$.
So $E_p$ is nearly spectral for large $p$.
\end{example}

\subsection{An exact result for Cartesian products}
\label{subsec:decomposeExact}

Let $G=H \oplus K$. In this subsection, we study sets $E$ of the form $E_1 \times K$ where $E_1 \subseteq H$.
We know that such a set $E$ is not further from being spectral than $E_1$: by Corollary \ref{cor:bdQsetprodspec}, $\widetilde{Q}(E) \leq \widetilde{Q}(E_1)$ for any normalized tightness quantity $\widetilde{Q}$.
We now prove the surprising converse: $\widetilde{Q}(E)=\widetilde{Q}(E_1)$, that is, they are exactly as far from being spectral as one another (Thm. \ref{thm:dimexpand}).

Results of this type are important for the following reason.
The product question for spectral sets asks: if $E_1 \times E_2$ is spectral, are $E_1$ and $E_2$ necessarily spectral?
This question is still open in general.
Our Theorem \ref{thm:dimexpand} implies that for any $E_1 \subseteq H$, $E_1$ is spectral if and only if $E_1 \times K$ is spectral.
Therefore, in the case of finite abelian groups, that proposition resolves this question when one factor is the whole group $K$.
In the case of $\R^d$, a similar result was obtained by Greenfeld and Lev \cite{GL}.
Specifically, they proved that if $I \subseteq \R$ is an interval and $\Sigma \subseteq \R^{d-1}$, then $\Sigma$ is spectral if and only if $I \times \Sigma$ is spectral.

Recall that we have the isomorphism $\widehat{G} \cong \widehat{H} \oplus \widehat{K}$, so we can think of $\widehat{H}$ as a direct summand of $\widehat{G}$.
The strategy for proving our result is as follows.
Let $E=E_1 \times K$, $E_1 \subseteq H$.
We first show that for any $B \subseteq \widehat{G}$ such that $(E,B)$ is a basis pair, $B$ has the same number of elements in each coset of $\widehat{H}$ (Prop. \ref{prop:basisequidistr}).
This allows us to apply Theorem \ref{thm:decomposeMultiTile} to the pair $(B,E)$; note that $E$ and $B$ appear here in the reverse of the usual order.
The conclusion is that the intersection of $B$ with some coset of $\widehat{H}$, properly translated, is a spectrally good pairing for $E_1$.
From this, we can obtain the desired result.

First, we prove the following Proposition \ref{prop:basisequidistr} on the structure of $B$ for which $(E,B)$ is a basis pair.

\begin{prop}
\label{prop:basisequidistr}
Let $E_1 \subseteq H$ and $E=E_1 \times K$.
Let $(E,B) \subseteq G \times \widehat{G}$ be a basis pair.
Then $B$ has exactly $\abs{E_1}$ elements in each coset of $\widehat{H}$.
\end{prop}

\begin{proof}
The proof is in the spirit of the proof of Proposition \ref{prop:almostinsubgp}.
Suppose otherwise. Let $\abs{E_1}=m$.
Then $B$ has at least $m+1$ elements in some coset
$\widehat{H}+\hat{k}$, $\hat{k} \in \widehat{K}$.
Let these elements be $\hat{h}_i+\hat{k}$, $1 \leq i \leq m+1$.
We claim that these elements are linearly dependent on $E$.

Because $\abs{E_1}=m$, there are $c_i \in \C$ such that
$\sum_{i=1}^{m+1} c_i \hat{h}_i(h) = 0$ for all $h \in E_1$.
So for any $(h,k) \in E_1 \times K$,
$$\sum_{i=1}^{m+1} c_i (\hat{h}_i+\hat{k})(h,k)
= \left(\sum_{i=1}^{m+1} c_i \hat{h}_i(h)\right)\hat{k}(k) = 0.$$
Thus, $\{\hat{h}_i+\hat{k}\}_{i=1}^{m+1}$ are linearly dependent on $E$, so that $B$ is not a basis for $E$.
This contradicts the hypothesis that $(E,B)$ is a basis pair.
\end{proof}

\begin{remark}
For a general $E=E_1 \times E_2$, the same argument as in the proof of Proposition \ref{prop:basisequidistr} implies that for any basis pair $(E,B)$, $B$ has at most $\abs{E_1}$ elements in each coset of $\widehat{H}$.
But if $E_2 \neq K$, ``at most'' here cannot be improved to ``exactly.''
\end{remark}

\begin{theorem}
\label{thm:dimexpand}
Let $E_1 \subseteq H$ and $E=E_1 \times K$. Then, for any normalized tightness quantity $\widetilde{Q}$,
$\widetilde{Q}(E)=\widetilde{Q}(E_1)$.
\end{theorem}

\begin{proof}
By Corollary \ref{cor:bdQsetprodspec}, $\widetilde{Q}(E) \leq \widetilde{Q}(E_1)$.
It remains to prove the reverse inequality.
Let $B \subseteq \widehat{G}$ be such that $\widetilde{Q}(E)=\widetilde{Q}_E(B)$.
Our goal is to construct $B_H \subseteq \widehat{H}$ such that
$\widetilde{Q}(E) \geq \widetilde{Q}_{E_1}(B_H)$, which will imply that $\widetilde{Q}(E) \geq \widetilde{Q}(E_1)$.

Observe that $(E,B)$ is a basis pair. Let $\abs{E_1}=\ell$. By Proposition \ref{prop:basisequidistr}, $B$ has exactly $\ell$ elements in each coset of $\widehat{H}$.
Let $\widehat{K}=\{\hat{k}_1,\dots,\hat{k}_m\}$, and let
$B_i = (B-\hat{k}_i) \cap \widehat{H} \subseteq \widehat{H}$.
We can now apply Theorem \ref{thm:decomposeMultiTile}, where the roles of $E$ and $B$ here are switched from that theorem.
We conclude that
$$\widetilde{Q}_B(E) \geq \min_{1 \leq i \leq m} \widetilde{Q}_{B_i}(E_1).$$
(For $\widetilde{Q} \neq \widetilde{D}$, we can strengthen the conclusion by replacing the minimum on the right-hand side with the maximum.)
So there is some $1 \leq i \leq m$ for which $\widetilde{Q}_B(E) \geq \widetilde{Q}_{B_i}(E_1)$.
By duality (Prop. \ref{prop:duality}), $\widetilde{Q}_E(B) \geq \widetilde{Q}_{E_1}(B_i)$.
This yields $\widetilde{Q}(E) \geq \widetilde{Q}(E_1)$, as desired.
\end{proof}

\begin{remark}
For a general $E=E_1 \times E_2$ where $E_2 \subseteq K$ is not necessarily $K$, the proof of Theorem \ref{thm:dimexpand} does not apply, because there is no analogue of Proposition \ref{prop:basisequidistr}.
\end{remark}

\section{Multi-tiling and geometric complexity}
\label{sec:MultiTileGeometricComplex}

In this section, we finally prove our main Theorem \ref{thm:mainUpperBdLevelComplexity} relating multi-tiling level and geometric complexity of a set to its tightness quantities. We start with a brief discussion of simultaneous bases in Section \ref{subsec:simulBasis} before heading to the main result in Section \ref{subsec:mainResult}.

\subsection{Simultaneous basis}
\label{subsec:simulBasis}

Let $G$ be a finite abelian group.
The concept of simultaneous bases for families of subsets of $G$, defined below, is crucial to our main Theorem \ref{thm:mainUpperBdLevelComplexity}.

\begin{definition}
For a family of subsets $E_1,E_2,\dots,E_m \subseteq G$ of equal size, we say that $B \subseteq \widehat{G}$ is a \emph{simultaneous basis} for this family if $(E_i,B)$ is a basis pair for all $i$.
\end{definition}

By Theorem \ref{thm:introRieszExist}, any family of one subset has a simultaneous basis.

The following example shows that not every family of subsets of $G$ of equal size has a simultaneous basis.
The $\Z_2^2$ case is due to Kolountzakis \cite{KoP}.

\begin{example}
\label{ex:nosimulbasis}
Consider the three subsets of $\Z_2^2$
$$E_1=\set{(0,0),(0,1)}, \quad E_2 = \set{(0,0),(1,0)},\quad E_3 = \set{(0,0),(1,1)}.$$
Note that these sets are subspaces of $\Z_2^2$.
Suppose that $B$ is a simultaneous basis for all the $E_i$.
By translational invariance (Prop. \ref{prop:invtrans}), assume that $B=\set{(0,0),b}$ for some $b \in \Z_2^2$.
By Proposition \ref{prop:almostinsubgp} in the case $\ell=0$, $b$ cannot be in $E_i^\perp$ for any $i$.
But $\cup_{i=1}^3 E_i^\perp = \Z_2^2$, a contradiction. So the $E_i$ do not have a simultaneous basis.

In general, let $p$ be a prime and consider $G=\Z_p^2$.
Let
$$d_1=(1,0),\ \ d_2 = (0,1),\ \ d_3=(1,1),\ \ d_4 = (1,2), \ \ \dots,\ \ d_{p+1}=(1,p-1)$$
be all $p+1$ directions in $G$.
For $1\leq i\leq m+1$, set $E_i=\set{nd_i: 0 \leq n \leq p-1}$.
Since $\cup_{i=1}^{p+1} E_i^\perp = G$, by an argument analogous to above, the $E_i$ do not have a simultaneous basis.
\end{example}

We can ask: under what circumstances is a simultaneous basis guaranteed to exist for a family of subsets of $G$ of equal size?
Kolountzakis \cite{KoP} observed that if the group $G$ is cyclic, then such existence is guaranteed, as the proposition below shows.

\begin{prop}
\label{prop:cyclicsimulbasis}
Let $G=\Z_m$ for some $m \geq 1$. Then any family of subsets of $G$ of equal size has a simultaneous basis.
\end{prop}

\begin{proof}
Let $B=\set{0,1,\dots,k-1} \subseteq \widehat{G}$, under the canonical identification.
We show that for any $E\subseteq G$ of size $k$, $(E,B)$ is a basis pair.
Writing $E=\set{x_0,x_1,\dots,x_{k-1}}$, $T(E,B)$ is the Vandermonde matrix
$$[\exp(2\pi i/m \cdot x_a b)]_{0 \leq a,b \leq k-1}.$$
Because $\exp(2\pi i/m \cdot x_a)$, $0 \leq a \leq k-1$, are distinct, the Vandermonde determinant of this matrix is nonzero, so that $(E,B)$ is a basis pair.
Hence, $B$ is a simultaneous basis for the family of all subsets of $G$ of size $k$.
\end{proof}

The condition for when a family of subsets admits a simultaneous basis may be worth studying further.

\subsection{Main result}
\label{subsec:mainResult}

In this subsection, we prove the following main result (Thm. \ref{thm:mainUpperBdLevelComplexity}) relating multi-tiling level and geometric complexity of a set to its spectral behavior.
For any $\ell,k \in \Z_{>0}$, there is a number $q(\ell,k)$ with the following property.
For any finite abelian group $G$, subgroup $H \subseteq G$, and subset $E \subseteq G$ satisfying the hypotheses:
\begin{enumerate}
\item $E$ multi-tiles $G$ by $H$ at level at most $\ell$;
\item there are at most $k$ distinct translated $H$-cross sections of $E$ up to $H$-translation; and
\item these translated $H$-cross sections admit a simultaneous basis in $\widehat{H}$,
\end{enumerate}
$\widetilde{Q}(E) \leq q(\ell,k)$ for any normalized tightness quantity $\widetilde{Q}$.

We may regard $k$ as the ``geometric complexity'' of $E$.
Such sets $E$ generalize tiles by subgroups ($\ell=k=1$) as in Proposition \ref{prop:TilesAreSpec}.

We conjecture that the second and third hypotheses are necessary; further insight is given in Examples \ref{ex:crossSectNoSimulBasis} and \ref{ex:crossSectBadSimulBasis}.

The strategy to proving our main result is as follows.
Using the ``looping'' ideas of Section \ref{subsec:boundsUpperE}, a good simultaneous basis $B_H$ for the translated $H$-cross sections of $E$ may be constructed (Prop. \ref{prop:lowQsimul}).
Then, Theorem \ref{thm:decomposeMultiTile} implies that $E$ pairs well spectrally with $B := \widetilde{B_H} \cdot H^\perp$. (See the setup in Section \ref{subsec:decomposeMultiTile}.)

We start with the proposition for the first step.

\begin{prop}
\label{prop:lowQsimul}
Let $G$ be a finite abelian group and $M$ be the minimal exponent of $G$.
Let $E_1,\dots,E_m \subseteq G$ have equal size $n>1$, and let $B \subseteq \widehat{G}$ be a simultaneous basis for the $E_i$. Then there is a $1 \leq k \leq M$ that is relatively prime to $M$ such that
$$\prod_{i=1}^m D_{E_i}(kB) \geq 1.$$
Consequently, for this $k$,
\begin{align*}
\min_{1 \leq i \leq m} L_{E_i}(kB) &>
\paren{\frac{n-1}{n}}^{n-1} \frac{1}{n^{mn-1}}
> \frac{1}{e n^{mn-1}}, \\
\max_{1 \leq i \leq m} U_{E_i}(kB) &<
n^2-\frac{n-1}{n^{m-1+\frac{m+1}{n-1}}}
\leq n^2-\frac{n-1}{2^{m+1}n^{m-1}},\\
\max_{1 \leq i \leq m} \rho_{E_i}(kB) &< 4n^{mn}.
\end{align*}
\end{prop}

\begin{proof}
By arguments involving symmetric polynomials in Lemma \ref{lem:lowerbdrootunity} and Theorem \ref{thm:upperbdcondeb}, for each $i$,
$$\mathbb{Z}_{>0} \ni \prod_{1 \leq k \leq M,\gcd(k,M)=1} D_{E_i}(kB) \geq 1.$$
Taking the product for all $i$, we obtain
$$\prod_{1 \leq k \leq M,\gcd(k,M)=1} \paren{\prod_{i=1}^m D_{E_i}(kB)} \geq 1,$$
whence the first result.

Now let $k$ be such that $\prod_{i=1}^m D_{E_i}(kB) \geq 1$.
By Proposition \ref{prop:simplebds}, $D_{E_i}(kB) \leq n^{n/2}$ for each $1 \leq i \leq m$.
Therefore, for each $i$,
$$D_{E_i}(kB) \geq \frac{1}{\prod_{1 \leq j \leq m, j \neq i}D_{E_j}(kB)} \geq n^{n(1-m)/2}.$$
Using Proposition \ref{prop:invertupperbdD} and calculating as in Proposition \ref{prop:upperBdPairLoopAround}, we obtain the desired bounds.
\end{proof}

We now turn to our main result.

\begin{theorem}
\label{thm:mainUpperBdLevelComplexity}
Let $H \subseteq G$ be a subgroup.
Let $E \subseteq G$ multi-tile $G$ with partner $H$ at level $\ell>1$. Suppose that, up to $H$-translation, there are $k$ distinct translated $H$-cross sections of $E$.
Assume further that these translated $H$-cross sections have a simultaneous basis in $\widehat{H}$. Then
$$\widetilde{L}(E) < e\ell^{k\ell}, \quad
\widetilde{U}(E) < \ell-\frac{\ell-1}{2^{k+1}\ell^k},
\quad \rho(E) < e \ell^{k\ell+1}, \quad
\widetilde{D}(E) \leq \sqrt{\ell}.$$
In particular, all normalized tightness quantities have upper bounds that depend only on $\ell$ and $k$.
\end{theorem}

\begin{proof}
Let $F_1,F_2,\dots,F_k \subseteq H$ be the distinct translated $H$-cross sections of $E$.
Let $B \subseteq \widehat{H}$ be their simultaneous basis.
By Proposition \ref{prop:lowQsimul}, there is an $s$ such that
$$\prod_{i=1}^k D_{F_i}(sB) \geq 1, \quad \min_{1\leq i \leq k} L_{F_i}(sB) > \frac{1}{e \ell^{k\ell-1}}, \quad
\max_{1 \leq i \leq k} U_{F_i}(sB) < \ell^2-\frac{\ell-1}{2^{k+1}\ell^{k-1}}.$$
Theorem \ref{thm:decomposeMultiTile} now yields the conclusion by pairing $E$ with $\widetilde{sB} \cdot H^\perp$.
\end{proof}

\begin{remark}
In the case $\ell=1$, which is not covered by Theorem \ref{thm:mainUpperBdLevelComplexity}, we know that $E$ is spectral by Proposition \ref{prop:TilesAreSpec}.
\end{remark}

\begin{remark}
By taking $H=G$ in Theorem \ref{thm:mainUpperBdLevelComplexity}, we recover the results of Theorem \ref{thm:upperBdSetBySize} that $\widetilde{Q}(E)$ have upper bounds depending only on $\abs{E}$.
Note that Theorem \ref{thm:upperBdSetBySize} gives a stronger bound for $\rho(E)$.
\end{remark}

As an application, 
Theorem \ref{thm:mainUpperBdLevelComplexity} shows that the sets in the example below behave well spectrally.

\begin{example}
\label{ex:twoCrossSects}
Let $m \geq 3$ and let $G=\Z_m^2$. Let
$$E_m=\paren{\set{0,1} \times (\Z_m \setminus \set{0})} \cup \set{(0,0),(2,0)}.$$
See Figure \ref{fig:complexFamilyBdRieszRatios}.
With $H=\Z_m \times \set{0}$, $E_m$ multi-tiles $G$ by $H$ at level $\ell=2$ and has $k=2$ distinct translated $H$-cross sections $\set{(0,0),(1,0)},\set{(0,0),(2,0)}$.
By Proposition \ref{prop:cyclicsimulbasis}, the translated $H$-cross sections have a simultaneous basis in $\widehat{H}$.
Therefore, Theorem \ref{thm:mainUpperBdLevelComplexity} yields
$$\rho(E_m) < 32e,$$
bounded independent of $m$.

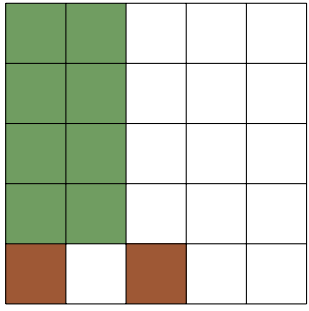
\begin{figure}
\begin{center}
\begin{tikzpicture}[scale=0.8]
\fill[Brown!70!white] (0,0) rectangle (1,1);
\fill[Brown!70!white] (2,0) rectangle (3,1);
\fill[OliveGreen!70!white] (0,1) rectangle (2,5);
\draw (0,0) grid (5,5);
\end{tikzpicture}
\end{center}
\caption{A set in the family in Example \ref{ex:twoCrossSects} with bounded Riesz ratios.}
\label{fig:complexFamilyBdRieszRatios}
\end{figure}
\end{example}

Compare Example \ref{ex:twoCrossSects} with Examples \ref{ex:cartclosespec} and \ref{ex:multiTileNearlySpec}.

We now present two examples that do not satisfy the second and third hypotheses of Theorem \ref{thm:mainUpperBdLevelComplexity}, respectively, in order to provide some insight into whether these hypotheses are necessary.

\begin{example}
\label{ex:crossSectNoSimulBasis}
Let $p \geq 3$ be a prime, $G_p=\Z_2^2 \times \Z_p$, and
$H=\Z_2^2 \times \set{0}$.
Let $F_i \subseteq H$ be the three subsets in Example \ref{ex:nosimulbasis}, under the appropriate identification.
Let
$$E_p = \paren{F_1 \times \set{0}} \cup \paren{F_2 \times \set{1}} \cup \paren{F_3 \times \set{2,3,\dots,p-1}}.$$
Then $E_p$ multi-tiles $G_p$ by $H$ at level $\ell=2$ and has $k=3$ distinct translated $H$-cross sections $F_i$, but the $F_i$ do not admit a simultaneous basis in $\widehat{H}$. Thus, Theorem \ref{thm:mainUpperBdLevelComplexity} does not apply, and we do not know whether $\rho(E_p)$ is bounded independent of $p$.

In fact, for any $B_p \subseteq \widehat{G_p}$ of the form $B_p = (B_p)_H \times \widehat{\Z_p}$, $(B_p)_H \subseteq \widehat{H}$, Theorem \ref{thm:decomposeMultiTile} implies that $(E_p,B_p)$ are not basis pairs.
So it is possible that that $\rho(E_p) \to \infty$ as $p \to \infty$.
\end{example}

\begin{example}
\label{ex:crossSectBadSimulBasis}
Let $p \geq 3$ be a prime, $G_p=\Z_p^2$, and $H_p=\Z_p \times \set{0}$. Set
$$E_p=\paren{\set{0} \times \Z_p} \cup \set{(1,0)} \cup \set{(x,x) : 1 \leq x \leq p-1}.$$
Then $E_p$ multi-tiles $G_p$ by $H_p$ at level $\ell=2$ but has $k=(p-1)/2$ distinct translated $H_p$-cross sections (which admit a simultaneous basis by Proposition \ref{prop:cyclicsimulbasis}).
Since $k$ is not bounded, we cannot conclude by Theorem \ref{thm:mainUpperBdLevelComplexity} whether $\rho(E_p)$ is bounded.

We conjecture that $\rho(E_p) \to \infty$ as $p \to \infty$.
In fact, for any $B_p \subseteq \widehat{G_p}$ of the form $B_p = (B_p)_H \times \widehat{\Z_p}$, $(B_p)_H \subseteq \widehat{H}$, we show that
$\rho_{E_p}(B_p) \to \infty$ as $p \to \infty$.
By translational invariance (Prop. \ref{prop:invtrans}),
we can assume $(B_p)_H=\set{0,b_p}$, where $1 \leq b_p \leq p-1$.
The translated $H_p$-cross sections of $E_p$ are, up to $H_p$-translation, $\set{0,a}$ for every $1 \leq a \leq p-1$. So by Theorem \ref{thm:decomposeMultiTile},
\begin{align*}
\rho_{E_p}(B_p) &\geq \max_{1 \leq a \leq p-1} \rho_{\set{0,a}}(\set{0,b_p})
\geq \rho_{\set{0,b_p^{-1}}}(\set{0,b_p}) \\
&= \rho_{\set{0,1}}(\set{0,1}) \to \infty
\end{align*}
as $p \to \infty$, using invariance under $\Z$-linear transformation (Prop. \ref{prop:invlineareb}) and Example \ref{ex:ebfarspec}.
\end{example}

\section{Conjectures}
\label{sec:conj}

\subsection{Continuity of the Riesz ratio}
\label{subsec:continuityRiesz}

In this subsection, we investigate continuity properties of the Riesz ratio. The motivation for this is the potential lifting of our results to construct a subset $E$ of $\R^d$ with no Riesz basis.
Specifically, we might aim to construct a sequence of sets $E_i \subseteq \R^d$ with large Riesz ratios that converge in some sense to a set $E \subseteq \R^d$.
(We define the Riesz ratio $\rho_E(B)$ for a pair $(E,B)$ to be the ratio between the optimal lower and upper Riesz constants,
and $\rho(E)=\inf_B \rho_E(B)$.)
If  $\rho(E_i) \to \infty$, then we might hope that this will imply that $\rho(E)=\infty$, i.e., $E$ has no Riesz basis.

In order to run this argument, the Riesz ratio must at least be upper semi-continuous in the sense that $\rho(E) \geq \limsup_{i \to \infty} \rho(E_i)$.
Nevertheless, we will show that, for a specific example in the context of finite abelian groups, the Riesz ratio is instead \emph{lower} semi-continuous at points of discontinuity.
Thus, the continuity seems to be ``going the wrong way."
If similar behaviors occur in $\R^d$, then there is a potential difficulty in using approximation to construct subsets of $\R^d$ with no Riesz basis.

Because our setting of finite abelian groups is discrete, we will first extend the notion of Riesz ratio to subsets with density, called \emph{generalized subsets}.
Then we will compute Riesz ratios of generalized subsets of $\Z_2$ and show that the Riesz ratio is lower semi-continuous in this case.

Let $G$ be a finite abelian group. A \emph{generalized subset} $E$ of $G$ is a function $E:G \to \R_{\geq 0}$.
This can be viewed as the subset $\supp E \subseteq G$ where each element $x \in \supp E$ has ``density'' $E(x)$.
An ordinary subset $E \subseteq G$ corresponds to the generalized subset $1_E$.
A generalized subset $E$ of $G$ induces the measure $\mu_E$ on $G$ defined by $\mu_E(F) = \sum_{x \in F} E(x)$ for any $F \subseteq G$.
Let $L^2(E)$ be the space of functions $f: G \to \C$ with the norm
$$\norm{f}_{L^2(E)}^2 = \int_G \abs{f}^2 \, d\mu_E = \sum_{x \in G} \abs{f(x)}^2E(x),$$
where we identify two functions that agree a.e. Hence, $L^2(E)$ is a vector space of dimension $\abs{\supp E}$.

A generalized subset $B$ of $\widehat{G}$ can be analogously defined.
Let $E$ and $B$ be generalized subsets of $G$ and $\widehat{G}$, respectively.
We say that $(E,B)$ is an \emph{equal-size pair} if $L^2(E)$ and $L^2(B)$ have equal dimension, that is, $\abs{\supp E}=\abs{\supp B}$.
Analogously to \eqref{eq:TEBgeneral}, define $T(E,B):L^2(B) \to L^2(E)$ to be the linear operator
$$T(E,B)c = \sum_{\hat{g} \in \widehat{G}} c(\hat{g})B(\hat{g})\hat{g} \in L^2(E), \quad c \in L^2(B).$$
We call $(E,B)$ a \emph{basis pair} if $T(E,B)$ is invertible. Notice that a basis pair is always an equal-size pair.

Let $(E,B)$ be an equal-size pair.
Analogously to Definition \ref{def:cond}, define the \emph{Riesz ratio of $B$ with respect to $E$} to be
$$\rho_E(B) := \cond(T(E,B))^2=\norm{T(E,B)}^2\norm{T(E,B)^{-1}}^2,$$
where the norm is the induced operator norm and this quantity is defined to be $\infty$ when $T(E,B)$ is not invertible.
(Similarly, we can also define $L_E(B)$ and $U_E(B)$, but we will not use these quantities.)
Define the \emph{Riesz ratio of $E$} to be $\rho (E) := \inf_B \rho_E(B)$, where the infimum ranges over all $B$ such that $(E,B)$ is an equal-size pair.

We now investigate Riesz ratios of generalized subsets of $\Z_2$ in order to gain more insight into continuity properties of the Riesz ratio.

Let $E$ be a nonzero (``nonempty'') generalized subset of $\Z_2$, and let $x=E(0)$ and $y=E(1)$, where $(x,y) \neq (0,0)$.
If $x=0$ or $y=0$, then $\abs{\supp E}=1$.
So if $(E,B)$ is an equal-size pair in this case, then $\abs{\supp B}=1$.
We can check that $\rho_E(B)=1$, and so $\rho (E)=1$.

Suppose now that $x,y \neq 0$. Let $\widehat{\Z_2}=\{\hat{0},\hat{1}\}$, where $\hat{i}(j)=(-1)^{ij}$ for $i,j \in \{0,1\}$.
Let $B$ be a generalized subset of $\widehat{\Z_2}$, where $a=B(\hat{0})$ and $b=B(\hat{1})$ are both nonzero.
Let $c \in L^2(B)$ with $c_0=c(\hat{0})$ and $c_1=c(\hat{1})$.
We can compute
$$f(c_0,c_1):=\frac{\norm{T(E,B)c}_{L^2(E)}^2}{\norm{c}_{L^2(B)}^2}
= \frac{x(c_0a+c_1b)^2+y(c_0a-c_1b)^2}{ac_0^2+bc_1^2}.$$
Thus, $\rho_E(B)$ is the ratio of the supremum and infimum of $f(c_0,c_1)$ over all nonzero $(c_0,c_1) \in \C^2$.

By scale invariance, it suffices to consider the case $ac_0^2+bc_1^2=1$.
Let $\sqrt{a}c_0=\cos\theta$ and $\sqrt{b}c_1=\sin \theta$, for some $\theta\in\mathbb{R}$.
Then we must maximize and minimize
\begin{align*}
f(\theta)&=x(\sqrt{a}\cos\theta+\sqrt{b}\sin\theta)^2
+ y(\sqrt{a}\cos\theta-\sqrt{b}\sin\theta)^2 \\
&= \frac{(x+y)(a+b)}{2} + \frac{(x+y)(a-b)}{2}\cos 2\theta
+(x-y)\sqrt{ab}\sin 2\theta.
\end{align*}
The maximum and minimum are
$(x+y)(a+b)/2 \pm \sqrt{((x+y)(a-b)/2)^2+(x-y)^2ab}$.
So
$$\rho_E(B) = \frac{1+\sqrt{q}}{1-\sqrt{q}},
\quad q = 1-\frac{4ab}{(a+b)^2}\frac{4xy}{(x+y)^2}.$$
To minimize $\rho_E(B)$, we must minimize $q$ over all nonzero $(a,b) \in \C^2$.
This occurs when $a=b$, whence
$$\rho(E) = \frac{\max(x,y)}{\min(x,y)}.$$
So in the $(x,y)$-plane, $\rho(E)=\cot\tau$, where $\tau$ is the angle that the line from the origin to $(x,y)$
makes with the closer of the $x$-axis and the $y$-axis.

We observe the following.
On the line $y=x$, $\rho(E)$ is one.
As we rotate the line closer to the $x$-axis or the $y$-axis, $\rho(E)$ increases and approaches infinity.
However, exactly on the $x$-axis and the $y$-axis, $\rho(E)$ once again becomes one.
Thus, the Riesz ratio is discontinuous at points where $\abs{\supp E}$ changes, in a way that shows $\rho(E)$ is not upper semi-continuous. On the other hand, the above calculations prove $\rho(E)$ is lower semi-continuous, i.e.,
$\rho(E) \leq \liminf_{E_0 \to E} \rho(E_0)$.

As stated above, the fact that the Riesz ratio is only lower semi-continuous in this example suggests a potential difficulty in using approximation to prove that a set in $\R^d$ has no exponential Riesz basis. However, if the lower semi-continuity persists in $\mathbb{R}^{d}$, then this leaves open the possibility of constructing interesting limit sets which have an exponential Riesz basis but do not multi-tile.

\subsection{Fractals}
\label{subsec:fractals}

These are some definitions of fractals, some conjectured to have no Riesz exponential basis.

Let $m\geq 2$ and $d\geq 2$ be positive integers, and let $S$ be a nonempty subset of $\mathbb{Z}_{m}^{d}$. For each integer $k\geq 1$, define the inclusion mapping $i_{k}:\mathbb{Z}_{m}^{d}\to\mathbb{Z}_{m^k}^{d}$ by $i_{k}(x)=(x_{i}\mod m^{k})_{i=1}^{d}$. For each integer $k\geq 1$, we also define the expansion mapping $g_{k}:\mathbb{R}^{d}\to\mathbb{Z}_{m^k}^{d}$ by
\[
g_{k}(x)=\left(\lfloor m^{k}x_{i} \rfloor\mod m^{k}\right)_{i=1}^{d}.
\]
Define the ``fractal'' $S^{\ast}\subseteq [0,1)^{d}$ by
\[
S^{\ast}:=[0,1]^{d}\cap\bigcup_{\ell=1}^{\infty}g_{\ell}^{-1}\left(i_{\ell}(S)\right),
\]
define the ``pre-fractal" sets $S_{k}^{\ast}\subseteq [0,1)^{d}$ by
\[
S_{k}^{\ast}:=[0,1]^{d}\cap\bigcup_{\ell=1}^{k}g_{\ell}^{-1}\left(i_{\ell}(S)\right),
\]
and define the discrete sets $S_{k}\subseteq \mathbb{Z}_{m^{k}}^{d}$ by
\[
S_{k}=g_{k}\left(\bigcup_{\ell=1}^{k}g_{\ell}^{-1}\left(i_{\ell}^{-1}(S)\right)\right).
\]

For example, when $d=2$, $m=2$, and $S=\{(1,0)\}$, we have the fractal $S^{\ast}=\{(x,y)\in [0,1)^{2}:\text{ for some integer }k\geq 0, \text{ we have }1/2^{k}\leq x<1/2^{k-1}\text{ and }0\leq y<1/2^{k}\}$. Also, for example, the discrete set $S_{2}\subseteq\mathbb{Z}_{4}^{2}$ is given by $S_{2}=\{(1,0), (2, 0), (3, 0), (2,1), (3,1)\}$ here.

We also define, generalizing the intersection by Terence Tao of inverse images with sets of the form $[-R,R]^{d}$ with $R\geq 1$ which is ``large but bounded," the sets
\[
S_{R}^{\ast}:=\bigcup_{z\in\mathbb{Z}^{d}\cap [-R,R]^{d}}(z+S^{\ast})
\]
and
\[
S_{k, R_k}^{\ast}:=\bigcup_{z\in\mathbb{Z}^{d}\cap [-R_k,R_k]^{d}}(z+S_{k}^{\ast}).
\]

\begin{conj}
For each integer $m>1$ and integer $d>1$, there exists a nonempty set $S\subseteq\mathbb{Z}_{m}^{d}$ such that (1) there exist $R_{k}\geq 1$ with $R:=\sup_{k\geq 1}R_{k}<\infty$ satisfying
\[
\limsup_{k\to\infty}\rho(S_{k, R_k}^{\ast})\leq \rho(S_{R}^{\ast})
\]
for which some $\varepsilon_{k}>0$ with $\sup_{k\geq 1}\varepsilon_{k}<\infty$ exist satisfying
\[
\abs{\rho(S_{k, R_k}^{\ast})-\rho(S_{k})}<\varepsilon_{k};
\]
and (2) the sets $\{S_{k}\}_{k=1}^{\infty}$ satisfy $\rho(S_{k})\to\infty$ as $k\to\infty$. 
\end{conj}

\begin{cor}
The above conjecture implies that there exist bounded sets $S_{R}^{\ast}$ of positive Lebesgue measure such that $L^{2}(S_{R}^{\ast})$ has no exponential Riesz basis.
\end{cor}

\vspace{0.2in}
\noindent
Sam Ferguson, Division of Advanced Mathematics Applications, Metron, Inc., Ste 600, 1818 Library St, Reston, VA 20190, USA \\ 
\href{mailto:sjf370@nyu.edu}{sjf370@nyu.edu}

\vspace{0.2in}
\noindent
Azita Mayeli, Department of Mathematics, The City University of New York,  The Graduate Center and Queensborough, New York, NY, 10016, USA \\ 
\href{mailto:amayeli@gc.cuny.edu}{amayeli@gc.cuny.edu}
 
\vspace{.2in}
\noindent
Nat Sothanaphan \\
\href{mailto:natsothanaphan@gmail.com}{natsothanaphan@gmail.com} 


\begin{thebibliography}{99}

\bibitem{REU} C. Aten, B. Ayachi, E. Bau, D. FitzPatrick, A. Iosevich, H. Liu, A. Lott, I. MacKinnon, S. Maimon, S. Nan, J. Pakianathan, G. Petridis, C. Rojas Mena, A. Sheikh, T. Tribone, J. Weill and C. Yu, Tiling sets and spectral sets over finite fields, \emph{J. Funct. Anal.} \textbf{273} (2017), 8: 2547--2577. \url{https://arxiv.org/abs/1509.01090}

\bibitem{BHM} D. Barbieri, E. Hern\'{a}ndez and A. Mayeli, Lattice sub-tilings and frames in LCA groups,
\emph{C. R. Math.} \textbf{355} (2017), 2: 193--199. \url{https://www.sciencedirect.com/science/article/pii/S1631073X16302618}.

\bibitem{FMV} T. Fallon, A. Mayeli and D. Villano, The Fuglede Conjecture holds in $\mb{F}_p^3$
for $p = 5,7$, \emph{arXiv}: 1902.02936, 2019.
\url{https://arxiv.org/abs/1902.02936}

\bibitem{FarkMM} B. Farkas, M. Matolcsi and P. M\'{o}ra, On Fuglede's conjecture and the existence of universal spectra, \emph{J. Fourier Anal. Appl.} \textbf{12} (2006), 5: 483--494. \url{https://arxiv.org/abs/math/0612016}

\bibitem{FarkR} B. Farkas and S. G. R\'{e}v\'{e}sz, Tiles with no spectra in dimension 4, \emph{Math. Scand.} \textbf{98} (2006), 1: 44--52. \url{https://doi.org/10.7146/math.scand.a-14982}

\bibitem{FS} S. Ferguson and N. Sothanaphan, Fuglede's conjecture fails in 4 dimensions over odd prime fields, \emph{Discrete Math.} (2019), published online. \url{https://doi.org/10.1016/j.disc.2019.04.026}

\bibitem{Fug} B. Fuglede, Commuting self-adjoint partial differential operators and a group theoretic problem, \emph{J. Funct. Anal.} \textbf{16} (1974), 1: 101--121. \url{https://doi.org/10.1016/0022-1236(74)90072-X}

\bibitem{Fug2} B. Fuglede, Orthogonal exponentials on the ball, \emph{Expo. Math.} \textbf{19} (2001): 267--272. \url{https://doi.org/10.1016/S0723-0869(01)80005-0}

\bibitem{GnL1} R. Greenfeld and N. Lev, Spectrality and tiling by cylindric domains, \emph{J. Funct. Anal.} \textbf{271} (2016), 10: 2808--2821. \url{https://doi.org/10.1016/j.jfa.2016.04.021}

\bibitem{GnL2} R. Greenfeld and N. Lev, Fuglede's spectral set conjecture for convex polytopes, \emph{Anal. PDE} \textbf{10} (2017), 6: 1497--1538. \url{dx.doi.org/10.2140/apde.2017.10.1497}

\bibitem{GnL3} R. Greenfeld and N. Lev, Spectrality of product domains and Fuglede’s conjecture for convex polytopes, \emph{J. Anal. Math.}, \textbf{140} (2020): 409--441. \url{https://doi.org/10.1007/s11854-020-0092-9}

\bibitem{GL} S. Grepstad and N. Lev, Multi-tiling and Riesz bases, \emph{Adv. Math.} \textbf{252} (2014): 1--6. \url{https://doi.org/10.1016/j.aim.2013.10.019}

\bibitem{HJ} R. A. Horn and C. R. Johnson, \emph{Topics in Matrix Analysis}, Cambridge University Press, Cambridge, 1994.

\bibitem{IKT} A. Iosevich, N. Katz and T. Tao, The Fuglede spectral conjecture holds for convex planar domains, \emph{Math. Res. Lett.} \textbf{10} (2003), 5: 559--569. \url{http://dx.doi.org/10.4310/MRL.2003.v10.n5.a1}

\bibitem{ILLW} A. Iosevich, C. Lai, B. Liu and E. Wyman, Fourier Frames for Surface-Carried Measures, \emph{Int. Math. Res. Not.}, rnz318. \url{https://doi.org/10.1093/imrn/rnz318}

\bibitem{IMP} A. Iosevich, A. Mayeli and J. Pakianathan, The Fuglede Conjecture holds in $\mathbb{Z}_p \times \mathbb{Z}_p$, \emph{Anal. PDE} \textbf{10} (2017), 4: 757--764. \url{https://arxiv.org/abs/1505.00883}

\bibitem{KoS} M. N. Kolountzakis, The study of translational tiling with Fourier Analysis, \emph{Fourier Analysis and Convexity}, 131--187, Appl. Numer. Harmon. Anal., Birkh\"{a}user, Boston, MA, 2004.

\bibitem{Ko} M. N. Kolountzakis, Multiple lattice tiles and Riesz bases of exponentials, \emph{Proc. Amer. Math. Soc.}, \textbf{143} (2015): 741--747. \url{https://doi.org/10.1090/S0002-9939-2014-12310-0}

\bibitem{Ko2} M. N. Kolountzakis, Packing near the tiling density and exponential bases for product domains, \emph{Bull. Hellenic Math. Soc.}, \textbf{60} (2016): 97--109. \url{https://arxiv.org/abs/1606.02452}

\bibitem{KoP} M. N. Kolountzakis, personal communication, 2018.

\bibitem{KM2} M. N. Kolountzakis and M. Matolcsi, Tiles with no spectra, \emph{Forum Math.} \textbf{18} (2006), 3: 519--528. \url{https://arxiv.org/abs/math/0406127}

\bibitem{KM1} M. N. Kolountzakis and M. Matolcsi, Complex Hadamard matrices and the spectral set conjecture, \emph{Collect. Math. Vol. Extra} (2006), 281--291. \url{https://arxiv.org/abs/math/0411512}

\bibitem{Lab} I. \L aba, Fuglede's conjecture for a union of two intervals, \emph{Proc. Amer. Math. Soc.} \textbf{129} (2001), 10: 2965--2972. \url{http://www.jstor.org/stable/2668831}

\bibitem{LM} N. Lev and M. Matolcsi, The Fuglede conjecture for convex domains is true in all dimensions, \emph{arXiv}: 1904.12262, 2019.
\url{https://arxiv.org/abs/1904.12262}

\bibitem{Mat} M. Matolcsi, Fuglede's conjecture fails in dimension 4, \emph{Proc. Amer. Math. Soc.} \textbf{133} (2005), 10: 3021--3026. \url{https://arxiv.org/abs/math/0611936}

\bibitem{My} G. Myerson, How small can a sum of roots of unity be?, \emph{Amer. Math. Monthly} \textbf{93} (1986), 6: 457--459.
\url{https://www.jstor.org/stable/2323469}.

\bibitem{OU} A. M. Olevskii and A. Ulanovskii, \emph{Functions with Disconnected Spectrum: Sampling, Interpolation, Translates}, American Mathematical Society, Providence, RI, 2016.

\bibitem{Ro} J. J. Rotman, \emph{An Introduction to the Theory of Groups}, Springer, New York, NY, 4th ed, 1994.

\bibitem{Tao} T. Tao, Fuglede's conjecture is false in 5 and higher dimensions, \emph{Math. Res. Lett.} \textbf{11} (2004), 2: 251--258. \url{https://arxiv.org/abs/math/0306134}

\bibitem{You} R. M. Young, \emph{An Introduction to Nonharmonic Fourier Series}, revised ed., Academic Press, Orlando, FL, 2001.

\end{thebibliography}
\end{document}